\documentclass[a4paper, 12pt]{article}
\usepackage{anysize}
\marginsize{3,5cm}{2,5cm}{2,5cm}{2,5cm}
\usepackage{amssymb}
\usepackage{amsmath,amsbsy,amssymb,amscd}
\usepackage{t1enc}
\usepackage[cp1250]{inputenc}
\usepackage[british]{babel}
\usepackage[all]{xy}
\usepackage{color}\usepackage{amsfonts}
\usepackage{latexsym}
\usepackage{amsthm}
\usepackage{mathrsfs}
\usepackage{hyperref}
\usepackage{graphicx}

\usepackage{color}
\usepackage{caption}
\usepackage{subcaption}
\usepackage{enumerate}
\usepackage{indentfirst}

\usepackage{mathrsfs} 

\DeclareMathAlphabet{\mathpzc}{OT1}{pzc}{L}{it} 

\newtheorem{definition}{Definition}[section]

\newtheorem{proposition}[definition]{Proposition}
\newtheorem{theorem}[definition]{Theorem}

\newtheorem{remark}[definition]{Remark}
\newtheorem{lemma}[definition]{Lemma}

\def\R{\mathbb{R}}
\def\T{\mathbb{T}}

\def\Z{\mathbb{Z}}
\def\N{\mathbb{N}}

\def\Q{\mathbb{Q}}
\def\cB{\mathcal{B}}

\def\cC{\mathcal{C}}

\newcommand{\bea}{\begin{eqnarray}}
  \newcommand{\eea}{\end{eqnarray}}
  \newcommand{\beab}{\begin{eqnarray*}}
  \newcommand{\eeab}{\end{eqnarray*}}

  \newcommand{\be}{\begin{equation}}
  \newcommand{\ee}{\end{equation}}

\title{Rigidity of joinings for some measure preserving systems}
\author{Changguang Dong\footnote{Department of Mathematics, University of Maryland, College Park, MD 20742, USA, E-mail:dongchg@math.umd.edu}, Adam Kanigowski\footnote{Department of Mathematics, University of Maryland, College Park, MD 20742, USA, E-mail:adkanigowski@gmail.com}, Daren Wei\footnote{Department of Mathematics, The Pennsylvania State University, University Park, PA 16802, USA, E-mail:duw170@psu.edu  D. W. was partially supported by the NSF grant DMS-16-02409}}
\begin{document}
\baselineskip=14pt \maketitle
\begin{abstract}
We introduce two properties: strong R-property and $C(q)$-property, describing a special way of divergence of nearby trajectories for an abstract measure preserving system. We show that systems satisfying the strong R-property are disjoint (in the sense of Furstenberg) with systems satisfying the $C(q)$-property. Moreover, we show that if $u_t$ is a unipotent flow on $G/\Gamma$ with $\Gamma$ irreducible, then $u_t$ satisfies the $C(q)$-property provided that $u_t$ is not of the form $h_t\times\operatorname{id}$, where $h_t$ is the classical horocycle flow. Finally, we show that the strong R-property holds for all (smooth) time changes of horocycle flows and non-trivial time changes of bounded type Heisenberg nilflows.
\end{abstract}
\tableofcontents

\section{Introduction}

In this paper we study rigidity of joinings of measure preserving systems. Recall that a {\em joining} of two systems $(T,X,\mathcal{B},\mu)$ and $(S,Y,\mathcal{C},\nu)$ is a measure on $X\times Y$ invariant under $T\times S$ and whose projections on $X$ and $Y$ are respectively $\mu$ and $\nu$. Classifying joinings between given systems  is a classical and difficult problem in ergodic theory. Some of the most celebrated results on classifying joinings were established by M. Ratner (\cite{Rat3}, \cite{Rat2}, \cite{Rat4}) for certain algebraic systems: unipotent flows on finite volume homogeneous spaces induced by nilpotent translations on semisimple Lie groups. A direct corollary of the aforementioned work is that all joinings between unipotent systems have to be \emph{algebraic}. In particular, unless there is an algebraic reason, different unipotent systems are disjoint (the only joining is the product measure).

The main goal of this paper is to try to generalize the phenomena established by Ratner to more general class of measure preserving systems. One of the main feature of unipotent systems, established by M. Ratner for horocycle flows \cite{Rat2}, time changes of horocycle flows \cite{RatnerTimechange}, and generalized by D. Witte \cite{Wit} to general unipotent systems is a polynomial way of divergence of nearby orbits in well understood directions. For horocycle flows the aforementioned directions is just the flow direction and, for general unipotent systems, the divergence happens always along some direction from the centralizer of the flow. It is worth noticing that even though first discovered in the algebraic world, the above notions of divergence  make sense for an arbitrary measure preserving system. The (divergence) property observed by Ratner for horocycle flows is now called Ratner's property\footnote{Ratner proved it for horocycle flows \cite{Rat2} and its time changes \cite{RatnerTimechange}, Witte generalized it to general unipotent flows \cite{Wit}.} (see \cite{Thouvenot}), and it has been established for certain parabolic systems outside the algebraic world (see \cite{FK}, \cite{Fr-Lem}, \cite{Fr-Lem2}, \cite{AK2}, \cite{KK}). In this paper we formulate an abstract version of Witte's property (called $C(q)$-property) on divergence of unipotent orbits (see Definition \ref{def:cent}). We also introduce a stronger version of Ratner's property, strong R-property\footnote{Acronym for strong Ratner's property.} (Definition \ref{def:sr}). Our main disjointness theorem (Theorem \ref{thm:dis}) states that any flow with strong R-property is disjoint from any flow enjoying the $C(q)$-property. This theorem can be understood as an abstract version of Ratner's joinings rigidity theorem for unipotent flows. Moreover we show that if $u_t$ is a unipotent flow acting on $G/\Gamma$, where $\Gamma$ is irreducible then $u_t$ satisfies the $C(q)$-property unless $u_t$ is generated by $\begin{pmatrix}0&1\\0&0\end{pmatrix}\oplus\operatorname{id}\in\mathfrak{sl}(2,\R)\oplus \mathfrak g'$. Finally, we establish the strong R-property for all (smooth) time changes of horocycle flows, for horocycle flows acting on tangent spaces of surfaces with variable negative curvature (with uniform parametrization) and for (smooth) time changes of bounded type Heisenberg nilflows. To summarize all the above cases, our results can be stated as following (we refer to Section \ref{sec:basicDef}, Section \ref{sec:homogeneousflows} and Section \ref{sec:var} for notations):
\begin{theorem}\label{thm:main}
Let $u_t$ be a unipotent flow on homogeneous space $G/\Gamma$ generated by $U\in\mathfrak{g}$ with $G$ is a linear semisimple Lie group, $\Gamma\subset G$ is a cocompact lattice and   $\operatorname{GR}(U)>3$. Let
\begin{itemize}
\item $h_t^{\tau}$ be a time change of a horocycle flow on $\operatorname{SL}(2,\mathbb{R})/\Gamma$ where $\Gamma\subset\operatorname{SL}(2,\mathbb{R})$ is a cocompact lattice and $\tau\in C^1(\operatorname{SL}(2,\mathbb{R})/\Gamma)$, $\tau>0$;
\item $v_t$ be the horocycle flow on the tangent space of a compact surface with variable negative curvature with uniform parametrization;
\item $T_t^{\tau_1}$ be a time change of a bounded type Heisenberg nilflow with $\tau_1\in W^s(\mathbb{T}^2)$ for $s>\frac{7}{2}$.
\end{itemize}
Then $u_t$ is disjoint with $h_t^{\tau}$, $v_t$ and $T_t^{\tau_1}$, respectively.
\end{theorem}

\textbf{Acknowledgements:} The authors would like to thank Svetlana Katok and Mariusz Lema\'{n}czyk for their careful reading and useful comments on a preliminary version of the paper. We would also like to thank anonymous referee for many useful comments and insightful suggestions on how to improve the paper.

\section{Definitions}
In this section we will state some basic definitions that will be used throughout the paper.
\subsection{Joinings, time changes of flows, special flows}\label{sec:basicDef}
Let $\varphi_t:(X,\mathcal{B},\mu)\to(X,\mathcal{B},\mu)$ and $\psi_t:(Y,\mathcal{C},\nu)\to(Y,\mathcal{C},\nu)$ be two
ergodic flows on probability standard Borel spaces.
\begin{definition}[Joinings]
A joining $\rho$ of $\varphi_t$ and $\psi_t$ is a $\varphi_t\times\psi_t-$invariant measure such that $\rho(X\times
C)=\nu(C)$ and $\rho(B\times Y)=\mu(B)$ for any $C\in\mathcal{C}$ and $B\in\mathcal{B}$. The set of joinings of $\varphi_t$ and $\psi_t$ is denoted by
$J(\varphi_t,\psi_t)$ and the set of ergodic joinings is denoted by $J^{e}(\varphi_t,\psi_t)$.
\end{definition}
\begin{definition}[Disjointness]
We say that $\varphi_t$ and $\psi_t$ are disjoint, denoted $\varphi_t\bot\psi_t$, if $J(\varphi_t,\psi_t)=\{\mu\otimes\nu\}$ (equivalently $J^e(\varphi_t,\psi_t)=\{\mu\otimes\nu\}$).
\end{definition}

Let $\tau\in L^1(X,\mu)$, $\tau>0$.

\begin{definition}[Time change] The flow $\varphi_t^{\tau}$ is called a time change (or a reparametrization) of the flow $\varphi_t$ if
$$\varphi_t^{\tau}(x)=\varphi_{\alpha(x,t)}(x), \text{ for all }(x,t)\in M\times\mathbb{R},$$
where the cocycle $\alpha(x,\cdot)$ is uniquely defined by the condition that
\begin{equation}\label{eq:time-change}
\int_0^{\alpha(x,t)}\tau(\varphi_s(x))ds=t.
\end{equation}
The flow $\varphi_t^{\tau}$ preserves the measure $d\mu^\tau:=\frac{\tau(\cdot)}{\int_X \tau d\mu}d\mu$.
\end{definition}

\paragraph{Special flows.}
Let $\lambda$ be the Lebesgue measure on $\mathbb{R}$. Let $T:(X,\mathcal{B},\mu)\to(X,\mathcal{B},\mu)$  and let $\tau\in L^1(X,\mathcal{B},\mu)$, $\tau>0$. We define the $\Z$-cocycle for $\tau$, by
$$
S_n(\tau)(x)=\left\{\begin{array}{ccc}
\tau(x)+\ldots+\tau(T^{n-1}x) &\mbox{if} & n>0\\
0&\mbox{if}& n=0\\
-\left(\tau(T^nx)+\ldots+\tau(T^{-1}x)\right)&\mbox{if} &n<0.\end{array}\right.
$$
Then we define the {\em special flow} $T_t^\tau$ on $X^\tau:=\{(x,s)\;:\;x\in X,0\leq s<\tau(x)\}$ by
\begin{equation}\label{specfl}
  T_t^\tau(x,s)=(T^{N(x,s,t)}x,s+t-S_{N(x,s,t)}(\tau)(x)),
\end{equation}
where $N(x,s,t)\in\mathbb{Z}$ is unique such that
\begin{displaymath}
S_{N(x,s,t)}(\tau)(x)\leq s+t<S_{N(x,s,t)+1}(\tau)(x).
\end{displaymath}
The flow $T_t^\tau$ preserves the measure $\mu^{\tau}:=\mu\otimes \lambda$ restricted to $X^\tau$. Moreover, if $X$ is a metric space with metric $d$, then so is $X^\tau$ with the product metric which we denote by $d^\tau$, i.e.\ $d^{\tau}((x,s),(x',s'))=d(x,x')+|s-s'|$.

\subsection{Homogeneous flows}\label{sec:homogeneousflows}
Let $G$ be a Lie group with Lie algebra $\mathfrak{g}$ and let $\mu$ be the Haar measure on $G$. Let $\exp : \mathfrak{g} \to G$ denote the exponential mapping of the Lie algebra $\mathfrak{g}$ onto $G$. Let $\Gamma \subset G$ be a {\em lattice}. We define the {\em homogeneous space}, $M:=G\slash \Gamma$. For $W\in \mathfrak{g}$, the homogeneous flow on $M$ is given by
\be\label{eq:1par}
\phi^W_t(x\Gamma)=(\exp(tW)x)\Gamma.
\ee
The flow $\phi^W_t$ preserves the Haar measure on $M$ (locally given by $\mu$). In this paper, we will study {\em unipotent flows} and {\em Heisenberg} nilflows.

\subsubsection{Unipotent flows}
We now assume additionally that  $G$ is a semisimple and that $W\in \mathfrak{g}$ is {\em unipotent}, i.e.\ the operator $ad_W:\mathfrak{g}\to \mathfrak{g}$ given by $ad_W(X)=[W,X]$ satisfies $ad_W^k=0$ for some $k\in \N$. Then the flow $\phi^W_t$ given by \eqref{eq:1par} is called a unipotent flow. Let $d=d_{G/\Gamma}$ be a right invariant metric on $G/\Gamma$. We recall the following important lemma for unipotent elements in $\mathfrak{g}$:

\begin{lemma}[\cite{KVW}]\label{lem:chainbasis} Let $W \in \mathfrak{g}$ be a unipotent element. Then there exists a basis  $\{X_i^n\}_{i=0}^{m_n}$, $n=1,\ldots,K$ of $\mathfrak{g}$  such that $X_0^n$ is in the centralizer
of $W$ for $n=1,\ldots,K$, and
\begin{equation}
ad_W(X_i^n) = X_{i-1}^n \mbox{ for all }1 \leq i \leq m_n\mbox{ and } n=1,\ldots,K.
\end{equation}
\end{lemma}
The above basis is called a {\em chain basis} for $W$. We also associate the following number with a unipotent element $W$:
\be\label{eq:GR}
\operatorname{GR}(W):=\frac{1}{2}\sum_{n=1}^Km_n(m_n+1).
\ee
By Jacobson-Morozov theorem, $\operatorname{GR}(U)\geq 3$ for any unipotent $U$ (there always exists one chain of length $3$ coming from the $\mathfrak{sl}(2,\R)$-triple).
\subsubsection{Heisenberg flows}
The three dimensional Heisenberg group $G$ is given by
$$G:=\left\{\left(
                    \begin{array}{ccc}
                      1 & x & z \\
                      0 & 1 & y \\
                      0 & 0 & 1 \\
                    \end{array}
                  \right):x,y,z\in\mathbb{R}
\right\}.$$

Suppose $W$ is an element of the Lie algebra
$$\mathfrak{g}:=\left\{\left(
                    \begin{array}{ccc}
                      0 & a & b \\
                      0 & 0 & c \\
                      0 & 0 & 0 \\
                    \end{array}
                  \right):a,b,c\in\mathbb{R}
\right\}
$$ of $G$. The Heisenberg nilflow generated by $W$ is given by \eqref{eq:1par} on $G/\Gamma$ for a lattice $\Gamma\subset G$.
We will be mostly interested in smooth time changes of Heisenberg nilflows. For this, it will be convenient to work with a {\em special flow representation} of Heisenberg nilflows.
\smallskip

As shown in \cite{AFU}, every {\it ergodic} nilflow $\phi_t^W$ can be represented as a {\it special flow}, where the base transformation $T_{\alpha,\beta}:\T^2\to \T^2$ is given by $T_{\alpha,\beta}(x,y)=(x+\alpha,y+x+\beta)$ for $\alpha\in [0,1)\setminus \Q$, $\beta\in \R$ and under a constant roof function $f(x,y)=C_W>0$.  In this paper we consider {\em bounded type} nilflows, that is we say that $W\in \mathfrak{g}$ is of  bounded type if the corresponding $\alpha$ is of bounded type\footnote{Recall that $\alpha$ is of bounded type if there exists $C_\alpha>0$ such that $q_{n+1}\leq C_\alpha q_{n}$ for every $n\in \N$, where $\{q_n\}_{n=1}^{+\infty}$ is the sequence of denominators of convergents of $\alpha$.}


For a function $f\in L^1(\T^2), f>0$, let $T_t^{f,\alpha, \beta}$ denote the special flow over $T_{\alpha, \beta}$ and under $f$. It follows (see e.g.\ \cite{AFU}) that for every $\tau\in L^1(G/\Gamma)$, $\tau>0$, the time change $T_t^{W,\tau}$ of the Heisenberg nilflow generated by $W\in\mathfrak g$ is isomorphic to a special flow $T_t^{f_{\tau},\alpha,\beta}$, where the roof function $f_\tau$ is as smooth as $\tau$. Throughout the paper we will consider the flow $T_t^{f_\tau,\alpha,\beta}$ where $\alpha$ is of bounded type. To shorten the notation, we will for simplicity denote such flows by $T_t^\tau$ and call them flows of {\em bounded type}. In fact, we will concentrate on the case that $\tau\in W^s(\mathbb{T}^2)$ with $s>\frac{7}{2}$, where $W^s(\mathbb{T}^2)$ is the standard Sobolev space.

\subsection{Horocycle flows over
compact surfaces of variable
negative curvature}\label{sec:var}
Let $S$ be a compact, negatively curved, oriented surface, and $g_s$ be the corresponding geodesic flow on its unit tangent bundle, $M$. There exists a 1-dimensional unstable foliation $W^u(x)$ and 1-dimensional stable foliation $W^s(x)$ for any $x\in M$. Since $S$ is oriented, the leaves $W^u(x)$ and $W^s(x)$ can be given an orientation. We wish to define a continuous 
flow $u_t$ whose orbits are exactly $W^u(x)$, but such a flow will depend on the way we parameterize each leaf. We will study a very special parametrization (called Margulis parametrization \cite{Margulis}, or uniform parametrization).  Such parametrization was first studied by B.\ Marcus in \cite{Marcus} and later by J.\ Feldman and D.\ Ornstein in \cite{Fel-Orn}. Namely,  let $v_t$, $k_t$ be two flows along respectively $W^u$ and $W^s$ such that
\[ g_sv_tx = v_{e^{s}t}g_sx \mbox{ for every } t,s\in \R \mbox{ and } x\in M.\]
and
\[ g_sk_tx = k_{e^{-s}t}g_sx \mbox{ for every } t,s\in \R \mbox{ and } x\in M.\]

It then follows that there exists a measure (Margulis measure) $\mu$ on $M$ which is preserved by both $v_t$ and $k_t$. We remark that in most cases, this action of $v_t$ is only H\"older, and is not even generated by a vector field (and the measure $\mu$ is singular with respect to the volume). It is worth pointing out that this parametrization is very important in our proof that $v_t$ is disjoint with $u_t$ since we need the parametrization to establish the control of the divergence along the orbit direction.

\section{Disjointness criterion}
In this section we introduce our main disjointness criterion. In what follows, we consider measure preserving flows $T_t$ on $(X,\cB,\mu,d)$, where $X$ is a $\sigma$-compact metric space (with the metric $d$), $\cB$ is the Borel $\sigma$-algebra and $\mu$ is a probability measure on $X$. We make the following standing assumption for all flows that we consider:
\be\label{eq:contin}
\forall_{\eta>0}\;\exists_{\eta'>0}\text{ such that }\forall_{|t|<\eta'}\forall_{x\in X} \;\;d(T_tx,x)<\eta.
\ee
Notice that the above condition is satisfied for any continuous flow if the space $X$ is additionally compact.

First we introduce a strong version of Ratner's property. Recall that R-property in the context of horocycle flows used by Ratner can be formulated as following: for every $\epsilon>0$, there exists $\kappa>0$ such that if $x$ and $x'$ are sufficiently close and not on the same orbit of horocycle flow, then there exists $M$, which depends on $x$ and $x'$, such that $d(T_tx,T_{t+p}x')<\epsilon$ with $p=\pm1$ for $t\in[M,(1+\kappa)M]$. The reason we created strong R-property instead of using R-property is that strong R-property enables us to control  the divergence on the {\bf whole} interval $[0,M]$ along the orbit. We don't have control on when the $C(q)$-property is seen for given points and hence the strong R-property allows us to see the divergence on the interval on which the $C(q$)-property is seen. This is essential in the proof of our disjointness criterion. In the following definition $\lambda$ denotes the Lebesgue measure on $\R$.
\begin{definition}[\mbox{Strong R-property}]\label{def:sr} Let $T_t$ be an ergodic flow on $(X,\cB,\mu,d)$ and let $p>0$. We say that $T_t$ has the {\em strong $R(p)$-property}, if for every $\epsilon>0$ and every $N\in\mathbb{N}$, there exists $\kappa=\kappa(\epsilon)\in (0,\epsilon)$, $Z\in \cB$, $\mu(Z)\geq 1-\epsilon$ and $\delta>0$ such that for every $x,x'\in Z$ with $d(x,x')<\delta$ and $x$ not in the orbit of $x'$ there exists $M_1\in \N$ with $M_1\geq N$ such that the following holds:
\begin{itemize}
\item[\textbf{R1}.] for every $L\in[\kappa^{-2}, M_1]$ there exists $p_L\in[-p,p]$ such that
$$
\lambda\left(\left\{t\in [L,L+\kappa L]\;:\;d(T_tx,T_{t+p_L}x')<\epsilon\right\}\right)\geq (1-\epsilon)\kappa L,
$$
and moreover $|p_M|\geq p/2$.
\end{itemize}
The flow $T_t$ is said to have the {\em strong R-property}, if it has the strong R(p)-property for some $p>0$.
\end{definition}

\begin{remark}
In the above definition of the strong R-property the condition $|p_M|\geq p/2$ can be modified to $|p_M|\geq C\cdot p$, where $C$ is an independent constant.
\end{remark}

\begin{remark}\label{rem:forback} Analogously to \cite{FK}, one can define a {\em switchable} version of the strong R-property by introducing the following condition (as an alternative to \textbf{R1}):
\item[\textbf{R2}.] for every $L'\in[\kappa^{-2}, M]$ there exists $p_{L'}\in[-p,p]$ such that
$$
\left|\left\{t\in [L',L'+\kappa L']\;:\;d(T_{-t}x,T_{-t+p_{L'}}x')<\epsilon\right\}\right|\geq (1-\epsilon)\kappa L',
$$
and moreover $|p_{L'}|\geq p/2$.

In this paper we focus on time changes of horocycle flows and Heisenberg nilflows for which (as proved in Section \ref{sec:rp}) one can show the strong R-property as stated in Definition \ref{def:sr}. The strong switchable version (for which, depending on $x,x'$, at least one of \textbf{R1} or \textbf{R2} is satisfied) might be used for flows satisfying the SWR-property (such as Arnol'd flows or von Neumann flows).
\end{remark}

We now introduce one of the main new definitions in the paper which describes a (parabolic) divergence for special directions. For a probability space $(Y,\cC,\nu,d)$ denote $UC(Y,d):=\{H\in Aut(Y,\cC,\nu)\;:\; H\text{ is uniformly continuous}\}$.

\begin{definition}\label{def:cent} Let $R_t$ be an ergodic flow on  $(Y,\cC,\nu,d)$. Let ${\bf D}\subset Aut(Y,\cC,\nu)\cap UC(Y,d)$ be compact in the uniform topology\footnote{Recall that for $V,W\in Aut(Y,\cC,\nu)$ the uniform metric is defined by $\bar{d}(V,W)=\sup_{y\in Y}d(Vy,Wy)$. We assume that ${\bf D}$ is compact for $\bar{d}$.} and let $q>0$. The flow $R_t$ has the {\em $C(\textbf{D},q)$-property} if there exists a compact\footnote{Compactness in uniform topology.} set ${\bf\tilde{D}}\subset {\bf D}\setminus \{Id\}$, a sequence $(A_k)_{k\in \N}\in Aut(Y,\cC,\nu)$, $A_k\to Id$ uniformly such that for every $\epsilon>0$ and $N\in \N$ there exists $\kappa=\kappa(\epsilon)>0$ and $\delta=\delta(\epsilon,N)>0$ such that for every  $k\in \N$ satisfying $d(A_k,Id)<\delta$, every $y\in Y$ and $y'=A_ky$ there exists $M_2\geq N$ such that
\begin{itemize}
\item[\textbf{C1}.] for every $L\in [\kappa^{-2}, M_2]$ there exists $S_{L}\in {\bf D}$ such that $S_{L}\circ R_r$ is ergodic for $r\in[-q,-q/2]\cup[q/2,q]$ and
$$\lambda\left(\left\{t\in [L,L+\kappa L]\;:\;d(R_ty,S_{L}R_{t}y')<\epsilon\right\}\right)\geq (1-\epsilon)\kappa L,$$
\item[\textbf{C2}.] $S_M\in {\bf\tilde{D}}$ and $S_M\circ R_r$ is ergodic for every $r\in [-q,q]$.
\end{itemize}

We say that $R_t$ has the {\em $C(q)$-property} if it has the $C({\bf{D}},q)$-property for some (compact) ${\bf{D}}\subset Aut(Y,\cC,\nu)\cap UC(Y,d)$.
\end{definition}

We will now state our main disjointness criterion:
\begin{theorem}\label{thm:dis}Let $T_t$ be a weakly mixing flow on $(X,\cB,\mu, d_1)$ with strong R-property and let $R_t$ be a flow on $(Y,\cC,\nu, d_2)$ satisfying the $C(q)$-property for every $q>0$. Then $T_t$ and $S_t$ are disjoint.
\end{theorem}

Some parts of the proof of Theorem \ref{thm:dis} follow similar steps to the proof of Theorem 5.9. in \cite{Fr-Lem2} and Theorem 3.1. in \cite{KLU}. We will provide a full proof for completeness. For $B\in \cC$ and $\epsilon>0$ let $V_\epsilon(B)=\{y\in Y:d_2(y,B)<\epsilon\}$. First, we need the following lemma:
\begin{lemma}[Lemma 5.6.\ in \cite{Fr-Lem2}]\label{lem:count} For every $B\in \cC$, we have $\{\epsilon>0\;:\; \nu(\partial V_\epsilon(B))>0\}$ is at most countable\footnote{Recall that for a set $B\in \cC$, $\partial B$ denote its boundary.}.
\end{lemma}

We will also use the following remark.
\begin{remark}[Remark 5.7.\ in \cite{Fr-Lem2}]\label{rem:reg}
Let $(X,d)$ be a Polish space and let $\mu$ be a (regular) probability measure on $(X,\cB)$. There exists a dense family $\{A_i\}_{i\geq 1}$ in $\cB$ such that $\mu(\partial A_i)=0$ for every $i\geq 1$.
\end{remark}

The main lemma in the proof of Theorem \ref{thm:dis} is the following\footnote{This lemma should be compared with Lemma 5.4. in \cite{Fr-Lem2}.}:
\begin{lemma}\label{lem:compact}Let $T_t$ and $R_t$ be two ergodic flows on respectively $(X,\cB,\mu,d_1)$ and $(Y,\cC,\nu,d_2)$. Let $p>0$ and ${\bf D}\subset Aut(Y,\cC,\nu)\cap UC(Y,d_2)$ be compact in the uniform topology. Take $\rho\in J^e(T_t,R_t)$ and let $A\in \cB$ and $B\in \cC$ with $\nu(\partial B)=0$. Then for every $\epsilon,\kappa,\delta>0$ there exists $N=N(\epsilon,\delta,\kappa)\in \R$ and a set $U\in \cB\otimes \cC$, $\rho(U)\geq 1-\delta$ such that for every $(x,y)\in U$ and every $L,M\geq N$ with $\frac{L}{M}\geq \kappa$, we have
$$
\left|\frac{1}{L}\int_{M}^{M+L}\chi_{A\times (R_{-r}\circ S^{-1})(B)}((T_t\times R_t)(x,y))-\rho(A\times(R_{-r}\circ S^{-1})(B))\right|<\epsilon,
$$
for every $(r,S)\in [-p,p]\times {\bf D}$.
\end{lemma}
\begin{proof}Fix $\epsilon,\delta,\kappa>0$. Since $\nu(\partial B)=0$, there exists $0<\epsilon'<\epsilon$ such that
\be\label{eq:nei}
\nu(V_{\epsilon'}(B))-\nu(B)< \epsilon/10,
\ee
and
\be\label{eq:nei2}
\nu(B)-\nu([V_{\epsilon'}(B^c)]^c)<\epsilon/10.
\ee

Let $Q\times {\bf D'}\subset [-p,p]\times {\bf D}$ be a finite set such that
\be\label{eq:cmo1}
\forall_{(r,S)\in [-p,p]\times {\bf D}}\exists_{(r',S')\in Q\times {\bf D'}}\text{ such that  } \sup_{y\in Y}d_2(S\circ R_ry,S'\circ R_{r'}y)<\epsilon'/10.
\ee
Notice that existence of $Q\times {\bf D'}$ follows from \eqref{eq:contin} and compactness of ${\bf D}$ (in uniform topology): we first pick a finite set  ${\bf D'}\subset {\bf D}$ such that for every $S\in {\bf D}$ there exists $S'\in {\bf D'}$ satisfying  $\bar{d}(S,S')=\sup_{y\in Y}d_2(Sy,S'y)<\epsilon'/100$. Then, since ${\bf D'}$ is finite and every element in ${\bf D'}$ is uniformly continuous there exists $\epsilon''>0$ such that for every $S'\in {\bf D'}$ and every $y,y'\in Y$ with $d_2(y,y')<\epsilon''$, we have $d_2(S'y,S'y')<\epsilon'/100$. Then, using \eqref{eq:contin}, we can pick a finite set $Q\in [-p,p]$ such that for every $r\in [-p,p]$ there exists $r'\in Q$ such that $\sup_{y\in Y} d_2(R_ry,R_{r'}y)<\epsilon''$. With this choice of parameters, we have
$$
d_2(S\circ R_ry,S'\circ R_{r'}y)\leq d_2(S\circ R_ry,S'\circ R_{r}y)+d_2(S'\circ R_{r}y,S'\circ R_{r'}y)<\epsilon'/50.
$$
This finishes the proof of existence of $Q\times {\bf D'}$.

For $(r,S)\in [-p,p]\times {\bf D}$ let $(r',S')\in Q\times {\bf D'}$ be such that \eqref{eq:cmo1} is satisfied.  By the construction of ${\bf D'}$ and \eqref{eq:cmo1}, for every $y\in Y$,
\be\label{eq:alu}
d_2(y,S'\circ R_{r'-r}S^{-1}y)\leq d_2(y,S'S^{-1}y)+d_2(S'S^{-1}y,S'R_{r'-r}S^{-1}y)\leq \epsilon'/10+\epsilon'/10,
\ee
(and analogous inequality holds for $S'$ and $S$ switched).
Therefore
$$
R_{-r}\circ S^{-1}(B)\subset R_{-r'}\circ S'^{-1}(V_{\epsilon'}(B))
$$
and analogously
$$
R_{-r'}\circ S'^{-1}(B)\subset R_{-r}\circ S^{-1}(V_{\epsilon'}(B)).
$$
The two above inclusion together with \eqref{eq:nei} imply that
\begin{multline}\label{eq:c000}
|\rho(A\times(R_{-r'}\circ S'^{-1})(V_{\epsilon'}(B)))-\rho(A\times(R_{-r}\circ S^{-1})(B))|\leq\\
|\rho(A\times(R_{-r'}\circ S'^{-1})(V_{\epsilon'}(B)))-\rho(A\times(R_{-r'}\circ S'^{-1})(B))|+\\
|\rho(A\times(R_{-r'}\circ S'^{-1})(B))-\rho(A\times(R_{-r}\circ S^{-1})(B))|\leq
\epsilon/3.
\end{multline}
Reasoning analogously and using \eqref{eq:nei2}, we also get
\be\label{eq:c001}
|\rho(A\times(R_{-r'}\circ S'^{-1})([V_{\epsilon'}(B^c)]^c))-\rho(A\times(R_{-r}\circ S^{-1})(B))|<\epsilon/3.
\ee
Let
\be\label{eq:tilc}
\tilde{V}(r',S'):=\left[A\times (R_{-r'}\circ S'^{-1})([V_{\epsilon'}(B^c)]^c)\right]^c.
\ee
We apply the ergodic theorem to (finitely many) functions $\chi_{A\times R_{-r'}\circ S'^{-1}(V_{\epsilon'}(B))}$ and $\chi_{\tilde{V}(r',S')}$, $(r',S')\in Q\times {\bf D'}$ to get that there exists $N=N(\epsilon, \kappa,\delta)$ and a set $U=U(\epsilon,\kappa,\delta)\in \cB\otimes \cC$ with $\rho(U)>1-\delta$ such that for every $(x,y)\in U$, every $L,M\geq N$ with $\frac{L}{M}\geq \kappa$ and every $(r',S')\in Q\times {\bf D'}$, we have (with $\bar{\delta}=\min\{\epsilon/10,\delta\}$)
\be\label{eq:bet2}
\left|\frac{1}{L}\int_{M}^{M+L}\chi_{A\times (R_{-r'}\circ S'^{-1})(V_{\epsilon'}(B))}((T_t\times R_t)(x,y))dt-\rho(A\times(R_{-r'}\circ S'^{-1})(V_{\epsilon'}(B)))\right|<\bar{\delta}
\ee
and
\be\label{eq:bet3}
\left|\frac{1}{L}\int_{M}^{M+L}\chi_{\tilde{V}(r',S')}((T_t\times R_t)(x,y))dt-\rho(\tilde{V}(r',S'))\right|<\bar{\delta}.
\ee
Fix $(r,S)\in [-p,p]\times {\bf D}$ and let $(r',S')\in Q\times {\bf D'}$ be such that \eqref{eq:cmo1} holds for $(r,S)$ and $(r',S')$. Notice that if $t\in \R$ is such that $(T_t\times R_t)(x,y)\in A\times (R_{-r}\circ S^{-1})(B)$ then by \eqref{eq:cmo1} (see also \eqref{eq:alu}), we have  $(T_t\times R_t)(x,y)\in A\times (R_{-r'}\circ S'^{-1})(V_{\epsilon'}(B))$.  Therefore by \eqref{eq:bet2} and \eqref{eq:c000}, we have
\begin{multline}\label{eq:N1}
\frac{1}{L}\int_{M}^{M+L}\chi_{A\times (R_{-r}\circ S^{-1})(B)}((T_t\times R_t)(x,y))dt\leq \\
\frac{1}{L}\int_{M}^{M+L}\chi_{A\times (R_{-r'}\circ S'^{-1})(V_{\epsilon'}(B))}((T_t\times R_t)(x,y))dt\leq \rho(A\times(R_{-r'}\circ S'^{-1})(V_{\epsilon'}(B)))+\bar{\delta}\leq\\
 \rho(A\times(R_{-r}\circ S^{-1})(B))+\bar{\delta} +\epsilon/3
\end{multline}
Similarly, if  $(T_t\times R_t)(x,y)\in [A\times (R_{-r}\circ S^{-1})(B)]^c$, then
$$(T_t\times R_t)(x,y)\in \tilde{V}(r',S').$$
Indeed, by \eqref{eq:tilc}, the above follows by showing
$$
(R_{-r'}\circ S'^{-1})([V_{\epsilon'}(B^c)]^c)\subset (R_{-r}\circ S^{-1})(B)
$$
Which is a straightforward consequence of \eqref{eq:alu}.
 Therefore by \eqref{eq:bet3} and \eqref{eq:c001}, we have
\begin{multline}\label{eq:N2}
\frac{1}{L}\int_{M}^{M+L}\chi_{A\times (R_{-r}\circ S^{-1})(B)}((T_t\times R_t)(x,y))dt=\\ 1- \frac{1}{L}\int_{M}^{M+L}\chi_{[A\times (R_{-r}\circ S^{-1})(B)]^c}((T_t\times R_t)(x,y))dt\geq \\
 1- \frac{1}{L}\int_{M}^{M+L}\chi_{\tilde{V}(r',S')}((T_t\times R_t)(x,y))dt\geq\\
  1-\rho(\tilde{V}(r',s'))-\bar{\delta}=\rho(A\times (R_{-r'}\circ S'^{-1})(V_{\epsilon}(B^c)^c))-\bar{\delta}\\
  \geq \rho(A\times(R_{-r}\circ S^{-1})(B)) -\bar{\delta}-2\epsilon/3.
\end{multline}
Summarizing, \eqref{eq:N1} and \eqref{eq:N2} together imply that
$$
\left|\frac{1}{L}\int_{M}^{M+L}\chi_{A\times (R_{-r}\circ S^{-1})(B)}((T_t\times R_t)(x,y))dt-\rho(A\times(R_{-r}\circ S^{-1})(B))\right|<\epsilon
$$
and this finishes the proof.
\end{proof}

We will now prove Theorem \ref{thm:dis}

\begin{proof}[Proof of Theorem \ref{thm:dis}] Let $\rho\in J^e(T_t,R_t)$ and $\rho\neq \mu\otimes \nu$. Let $p>0$ be such that $T_t$ satisfies the $R(p)$ property (see Definition \ref{def:sr}) and fix $q=p$. Let ${\bf D}={\bf D}(q)\subset Aut(Y,\cC,\nu)$ be such that $R_t$ satisfies the $C({\bf D},q)$ property (the existence of $\bf{D}$ follows from the assumptions of Theorem \ref{thm:dis} and Definition \ref{def:cent}). Let $\{A_i\}_{i\geq 1}$ and $\{B_i\}_{i\geq 1}$ be dense families in $\cB$ and $\cC$ respectively such that $\mu(\partial A_i)=\nu(\partial B_i)=0$ for every $i\geq 1$ (see Remark \ref{rem:reg}). Let
 $${\bf E'}:=\left[\big(([-p,-p/2]\cup [p/2,p])\times {\bf D}\big)\cup ([-p,p]\times \bf \tilde{D})\right]$$
and let
$$
{\bf E}:=\{(r,S)\in {\bf E'}\;:\; S\circ R_r \text{ is ergodic }\}.
$$

We consider the following function $\Psi:{\bf E} \to \R_+$
$$
\Psi(t,S)=\sum_{i,j=1}^{\infty}\frac{1}{2^{i+j}}\left|\rho(A_i\times (R_{-t}\circ S^{-1})(B_j))-\rho(A_i\times B_j)\right|.
$$
Notice that $\Psi$ is a continuous function\footnote{We consider the strong operator topology on ${\bf D}\subset Aut(Y,\cC,\nu)$. Recall also that $\bf{D}$ is compact, hence closed.}. Recall that by the definition of {\bf E}, for every $(r,S)\in \bf{E}$, $\Psi(r,S)>0$. Indeed, if $\Psi(r,S)=0$ then $\rho(A_i\times (R_{-r}\circ S^{-1})(B_j))=\rho(A_i\times B_j)$ for every $i,j\geq 1$. But since $\{A_i\}_{i\geq 1}$ and $\{B_i\}_{i\geq 1}$ are dense, it follows that $\rho(A\times (R_{-r}\circ S^{-1})(B))=\rho(A\times B)$ for every $A\in \cB$ and $B\in \cC$. This however contradicts to the ergodicity of $(S\circ R_r)^{-1}$ and $\rho\neq\mu\otimes\nu$ (recall that every ergodic transformation is disjoint from $\operatorname{Id}$). By compactness of ${\bf E}$ it follows that there exists $\epsilon_0>0$ such that
$\inf_{(r,S)\in {\bf E}}\Psi(r,S)>\epsilon_0$. This in turn implies that there exists $H\in \N$ such that for every $(r,S)\in {\bf E}$,
$$
\sum_{i,j\geq 1}^H\frac{1}{2^{i+j}}\left|\rho(A_i\times (R_{-r} \circ S^{-1} )(B_j))-\rho(A_i\times B_j)\right|\geq \epsilon_0/2.
$$
Summarizing,
\begin{multline}\label{eq:dsj}
\text{for every } (r,S)\in  {\bf E} \text{ there exists } i,j\in \{1,\ldots ,H\}
\text{ such that} \\
|\rho(A_i\times (R_{-r} \circ S^{-1})(B_j))-\rho(A_i\times B_j)|\geq \epsilon_0.
\end{multline}
 By Lemma \ref{lem:count} for $B=B_j$, $j\in\{1,\ldots,H\}$ it follows that there exists $0<\epsilon_1<\epsilon_0/16$ such that
$$
\nu(V_{\epsilon_1}(B_j)\setminus B_j)<\epsilon_0/20  \text{ for } j\in\{1,\ldots, H\}.
$$
Moreover using Lemma \ref{lem:count} again, by taking a smaller $\epsilon_1$ if necessary, we can also assume that
$$
\nu(V_{\epsilon_1}(A_i)\setminus A_i)<\epsilon_0/20  \text{ for } j\in\{1,\ldots, H\}.
$$
Since $\rho$ is a joining, this implies that
\begin{equation}\label{eq:su}
|\rho(A_i\times B_j)-\rho(V_{\epsilon_1}(A_i)\times V_{\epsilon_1}(B_j))|<\epsilon_0/10.
\end{equation}

Let moreover $\widetilde{V}_{ij}:=\left[[V_{\epsilon_1}(A_i^c)]^c\times [V_{\epsilon_1}(B_j^c)]^c\right]^c$. By taking still smaller $\epsilon_1$ if necessary we can assume that for $i,j\in \{1,\ldots H\}$, we have (compare with \eqref{eq:nei2})
\be\label{eq:su2}
|\rho(\widetilde{V}_{ij})-(1-\rho(A_i\times B_j))|<\epsilon_0/10.
\ee
Indeed, it remains to notice that $\rho(\widetilde{V}_{ij})=1-\rho([V_{\epsilon_1}(A_i^c)]^c\times [V_{\epsilon_1}(B_j^c)]^c)$.

Let $\kappa=\kappa(\epsilon_1)=\min\{\kappa_1,\kappa_2\}$, where $\kappa_1>0$ comes from the strong R-property and $\kappa_2>0$ comes from the $C(q)$-property (both for $\epsilon_1$). By ergodic theorem for (finitely many) functions $\chi_{A_i\times B_j}$, $\chi_{V_{\epsilon_1}(A_i)\times V_{\epsilon_1}(B_j)}$ and $\chi_{\tilde{V}_{ij}}$, for $i,j\in\{1,\ldots H\}$ we get that there exists $N_1\in \R$ and $Z_1\in \cB\otimes \cC$ with $\rho(Z_1)\geq 1-\epsilon_0$ such that for every $L,M\geq N_1$, $\frac{L}{M}\geq \kappa$ and for every $(x,y)\in Z_1$, we have
\be\label{eq:bet5}
\left|\frac{1}{L}\int_{M}^{M+L}\chi_{C_i\times D_j}((T_t\times R_t)(x,y))dt-\rho(C_i\times D_j)\right|<\epsilon_1/10
\ee
for $C_i\in  \{A_i,V_{\epsilon_1}(A_i)\}$, $D_j\in  \{B_j,V_{\epsilon_1}(B_j)\}$  and $C_i$ and $D_j$ satisfying $C_i\times D_j=\widetilde{V}_{ij}$ (recall that $\widetilde{V}_{ij}$ is a product set) for $i,j\in\{1,\ldots, H\}$.

By Lemma \ref{lem:compact} for $A_i$ and $B_j$ with $i,j\in\{1,\ldots,H\}$ and for $\epsilon_1$, $\delta=\epsilon_1/10$ and $\kappa$, we get that there exists $N_2\in \R$ and $Z_2\in \cB\otimes \cC$ with $\rho(Z_2)\geq 1-\epsilon_0$ such that for every $L,M\geq N_2$, $\frac{L}{M}\geq \kappa$ and for every $(x,y)\in Z_2$, we have
\be\label{eq:bet6}
\left|\frac{1}{L}\int_{M}^{M+L}\chi_{A_i\times (R_{-r}\circ S^{-1})(B_i)}((T_t\times R_t)(x,y))-\rho(A_i\times(R_{-r}\circ S^{-1})(B_j))\right|<\epsilon_1/10
\ee
for every $(r,S)\in [-p,p]\times {\bf D}$.

Let $Z'=Z_1\cap Z_2\cap (Z\times Y)$, where $Z$ is the set from the strong R-property (for $\epsilon_1>0$ and $N=\max\{N_1,N_2,\kappa^{-2}\}$). Notice that since $\mu(Z)\geq 1-\epsilon_1\geq 1-\epsilon_0$, we have $\rho(Z')\geq 1-10\epsilon_0$. Moreover, since $\rho$ is regular and $X\times Y$ is $\sigma$-compact, we can assume additionally that $Z'$ is compact. Consider the projection $\pi:X\times Y\to Y$, $\pi(x,y)=y$. Since $Z'$ is compact it follows that the fibers of the map $\pi_{|Z'}:Z'\to \pi(Z')$ are also $\sigma$-compact and $\pi(Z')$ is compact. Therefore, by Kunugui's selection theorem, \cite{Kun}, it follows that there exists a measurable selection $x:\pi(Z')\to X$ such that $(x(y),y)\in Z'$. Applying Egorov theorem to the function $x$, it follows that there exists $Y'\in Y$, $\nu(Y')\geq 1-\epsilon$ and such that $x:\pi(Z')\cap Y'\to X$ is uniformly continuous. This means that for  $\delta=\min\{\delta_1(\epsilon_1,N),\delta_2(\epsilon,N)\}$(where $\delta_1$ comes from the strong R-property and $\delta_2$ comes from the $C(q)$-property) there exists $\delta'>0$ such that for $y,y'\in Y'':=\pi(Z')\cap Y'$ with $\nu(Y'')\geq 1-\epsilon_0$ and  $d_2(y,y')<\delta'$, $d_1(x(y),x(y'))\leq \delta$. Let
$$
\tilde{Z}=Z'\cap (X\times Y'').
$$
Then $\rho(\tilde{Z})\geq 1-15\epsilon_0$.  Let $\tilde{Z}_Y=\pi(\tilde{Z})$. Then $\nu(\tilde{Z}_Y)\geq 1-15\epsilon_0$.
Let $(A_k)_{k\in \N}$ be the sequence of automorphisms coming from the $C(q)$-property. There exists $k_0=k_0(\epsilon_1)\in \N$ such that for $k\geq k_0$, we have
$$
\nu(A_{-k}(\tilde{Z}_Y)\cap \tilde{Z}_Y)>0.
$$
Therefore, there exists $(x,y)\in Z'$ and $(x',y'=A_ky)\in Z'$ such that $d_1(x,x')<\delta$ and $d_2(y,y')<\delta$. This, by the strong R-property implies that there exists $M_1=M_1(x,x')$ such that {\bf R1} holds for $x,x'$ and similarly there exists $M_2=M_2(y,y')$ such that {\bf C2.} holds for $y,y'$. Let $\tilde{M}=\min\{M_1,M_2\}\geq N$. For simplicity, denote $p_0=p_{\tilde{M}}$ (coming from {\bf R1}) and $S=S_{\tilde{M}}$ (coming from {\bf C1}). Recall that by the definition of $N$, we have $\tilde{M}\geq \kappa^{-2}=\max\{\kappa_1^{-2},\kappa_2^{-2}\}$. Since $\tilde{M}$ is the minimum of $M$ and $M'$ by {\bf R1} and {\bf C2} it follows that
\be\label{eq:ee}
\begin{aligned}
&\text{ either } p\geq |p_0|\geq p/2, S\in  {\bf D}\text{ and } (p_0,S)\in {\bf E};\\
&\text{  or  } \;\;p_0\in[-p,p], S\in {\bf \tilde{D}}\text{ and } (p_0,S)\in {\bf E}.
\end{aligned}
\ee

Notice that in the interval $I=[\tilde{M},\tilde{M}+\kappa\tilde{M}]$ by {\bf R1}, there exists $I_1\subset I$ with $\lambda(I_1)\geq (1-\epsilon_1)\kappa\tilde{M}$ such that
\be\label{eq:flowclose}
d_1(T_tx,T_{t+p_{0}}x')<\epsilon_1 \text{ for } t\in I_1.
\ee
Moreover by {\bf C1}, there exists $I_2\subset I$ with $\lambda(I_2)\geq (1-\epsilon_1)\kappa\tilde{M}$ such that
\be\label{eq:centclose}
d_2(R_ty,SR_ty')<\epsilon_1 \text{ for } t\in I_2.
\ee

 By \eqref{eq:flowclose}, for $t\in I_1\cap I_2$  if $(T_{t+p_0}x',R_{t+p_0}y')\in A_i\times R_{p_0}S^{-1}B_j$, we have
\be\label{eq:juk}
(T_tx, R_ty)\in V_{\epsilon_1}(A_i)\times V_{\epsilon_1}(B_j).
\ee
Moreover, if  $t\in I_1\cap I_2$,  $(T_{t+p_0}x',R_{t+p_0}y')\in [A_i\times R_{p_0}S^{-1}B_j]^c$, then
\be\label{eq:juk2}
(T_tx, R_ty)\in \left[[V_{\epsilon_1}(A_i^c)]^c\times [V_{\epsilon_1}(B_j^c)]^c\right]^c.
\ee
By \eqref{eq:bet6} (for $r=-p_0$ and $S$), \eqref{eq:juk} and \eqref{eq:bet5} (for $C_i=V_{\epsilon_1}(A_i)$ and $D_j=V_{\epsilon_1}(B_j)$), we have (for $L=\kappa \tilde{M}$)
\begin{multline}\label{mult0}
\rho(A_i\times R_{p_0}S^{-1}B_j)\leq \frac{1}{L}\int_{\tilde{M}}^{\tilde{M}+L}\chi_{A_i\times (R_{p_0}\circ S^{-1})(B_j)}((T_t\times R_t)(x',y'))dt+\epsilon_1/10\leq\\
\frac{1}{L}\int_{\tilde{M}}^{\tilde{M}+L}\chi_{A_i\times (R_{p_0}\circ S^{-1})(B_j)}((T_{t+p_0}\times R_{t+p_0})(x',y'))dt+\frac{2|p_0|}{L}+\epsilon_1/10\leq \\
\frac{|I_1^c\cup I_2^c|}{L}+\frac{2|p_0|}{L}+\epsilon_1/10+\frac{1}{L}\int_{\tilde{M}}^{\tilde{M}+L}\chi_{V_{\epsilon_1}(A_i)\times V_{\epsilon_1}(B_j)}((T_t\times R_t)(x,y))dt\\
 \leq 3\epsilon_1+\epsilon_1/5+\rho(V_{\epsilon_1}(A_i)\times V_{\epsilon_1}(B_j))\leq 4\epsilon_1+\epsilon_0/10 +\rho(A_i\times B_j),
 \end{multline}
 where the last inequality comes from \eqref{eq:su} and we also used $\frac{2|p_0|}{L}\leq \frac{2|p_0|}{\kappa \tilde{M}}\leq 2\kappa |p_0| \leq \epsilon_1/10$, which follows by $\tilde{M}\geq\kappa^{-2}$ and by taking smaller $\kappa$ if necessary.

Recall that $\widetilde{V}_{ij}=\left[[V_{\epsilon_1}(A_i^c)]^c\times [V_{\epsilon_1}(B_j^c)]^c\right]^c$. Analogously, by \eqref{eq:bet5}, \eqref{eq:bet6}, \eqref{eq:juk2} and $\tilde{M}\geq\kappa^{-2}$, we have
\begin{multline}\label{mult1}
\rho(A_i\times R_{p_0}S^{-1}B_j)\geq \frac{1}{L}\int_{\tilde{M}}^{\tilde{M}+L}\chi_{A_i\times (R_{p_0}\circ S^{-1})(B_j)}((T_t\times R_t)(x',y'))dt-\epsilon_1/10=\\
1-\epsilon_1/10-\frac{1}{L}\int_{\tilde{M}}^{\tilde{M}+L}\chi_{[A_i\times (R_{p_0}\circ S^{-1})(B_j)]^c}((T_t\times R_t)(x',y'))dt\geq \\
1-\frac{2|p_0|}{L}-\frac{|I_1^c\cup I_2^c|}{L}-\epsilon_1/10-\frac{1}{L}\int_{\tilde{M}}^{\tilde{M}+L}\chi_{\widetilde{V}_{ij}}((T_t\times R_t)(x,y))dt\geq\\
1-4\epsilon_1- \rho(\widetilde{V}_{ij})\geq \rho(A_i\times B_j)-4\epsilon_1-\epsilon_0/10,
 \end{multline}
 the last inequality by \eqref{eq:su2}.
Consequently, \eqref{mult0} and \eqref{mult1} imply that for every $i,j\in\{1,\ldots H\}$,
$$
|\rho(A_i\times (R_{p_0}\circ S^{-1})(B_j))-\rho(A_i\times B_j)|<\epsilon_0/3.
$$
This together with \eqref{eq:ee} contradicts \eqref{eq:dsj} (since $\epsilon_1<\epsilon_0/16$ and recall that $r=-p_0$).
\end{proof}

\section{Disjointness of unipotent flows and flows with 
strong R-property}
In this section we will prove that flows with strong R-property are disjoint from unipotent flows $u_t$ generated by $U$ acting on $(G\slash \Gamma,\mu)$ where $\Gamma$ is irreducible, $\mu$ is the Haar measure and $\operatorname{GR}(U)>3$ (see \eqref{eq:GR}). This result, by Theorem
\ref{thm:dis}, is an immediate consequence of the following:

\begin{proposition}\label{pro:cent} Let $u_t$ be a unipotent flow generated by $U\in \mathfrak{g}$ on $G\slash \Gamma$ where $\Gamma\subset G$ is an irreducible compact lattice. If $\operatorname{GR}(U)>3$, then $u_t$ satisfies the $C(q)$-property for every $q>0$.
\end{proposition}

\begin{remark}It follows by Lemma 3.9. in \cite{KVW2} that the only unipotent flows that satisfy $\operatorname{GR}(U)\leq 3$ are of the form $h_t\times\operatorname{id}$ on $(\operatorname{SL}(2,\R)\times G')/\Gamma$ where $G'$ is a subgroup of semisimple Lie group $G$, $\Gamma $ is an irreducible lattice in $\operatorname{SL}(2,\R)\times G'$ and $h_t$ is the classical horocycle flow.
\end{remark}

In order to prove Proposition \ref{pro:cent}, we need the following lemma.
\begin{lemma}\label{lm:ergo}
Let $X_n\mapsto X_{n-1}\mapsto \cdots\mapsto X_1\mapsto X_0$ be a chain of depth at least 2, and $\psi_t$ be the flow generated
by $X_0$ on $G/\Gamma$. Then for any $s_0,t_0\in \R$ with $t_0^2+s_0^2>0$, the automorphism $\psi_{t_0}u_{s_0}$ is ergodic.
\end{lemma}
\begin{proof}
Let $\{U,V,X\}$ be an $\mathfrak{sl}(2,\mathbb R)$ triple. Let $g_t$ be the flow generated by $X$ on $G/\Gamma$. Since
$[X, U]=2U$, the vector field generated by $U$ lies in the unstable bundle of $g_t$ for $t>0$, hence $g_t$ is a partially hyperbolic flow. Moreover (see e.g.\ \cite{Kri} Theorem $5.6$) it follows that $[X, X_0]=\lambda X_0$ for some $\lambda >0$. Therefore, $X_0$ is also a $g_t$ invariant vector field belonging to the unstable bundle of $g_t$.

By definition of the chain basis (and since $\psi_t$ is generated by $X_0$) it follows that the automorphisms $\psi_{t_0}$ and $u_{s_0}$ commute. Let $H_\Delta:=\{\exp(\ell[t_0X_0+s_0U])\;:\; \ell\in \Delta\}$ for $\Delta\in\{\Z,\R\}$. It follows that $H_\R$ is a proper subgroup of $G$, and $H_\Z$ is a discrete subgroup. So by Moore's ergodicity theorem, to prove ergodicity of $\psi_{t_0}u_{s_0}$ it is enough to show that the subgroup $H_\Z$ is not compact.

Since $H_\R\subset H_\Z\cdot\{\exp(\ell[t_0X_0+s_0U]): \ell\in [0,1]\}$, it suffices to prove $H_\R$ is not compact. To that end, we will argue by contradiction. Assume that $H_\R$ is compact, then for any $x\in G/\Gamma$, $H_x:=\{h\cdot x:h\in H_\R\}$ is an immersed compact submanifold inside the unstable manifold of $g_t$ in $G/\Gamma$. Moreover, $H_x$ depends continuously on $x$, and the family $\{H_x\}_{x\in G/\Gamma}$ is a $g_t$ invariant foliation (with compact leaves), and $g_t(H_x)=H_{g_t(x)}$ for any $x\in G/\Gamma$ and $t\in \R$. Since $g_t$ is ergodic, almost every point is recurrent, then pick a point $x$ such that, there exists a decreasing sequence $\{t_i:i\in\N, t_i\to -\infty\text{ as }i\to\infty\}$ with $g_{t_i}(x)\to x$ as $i\to\infty$. By continuity of $H_x$ on $x$, we have $g_{t_i}(H_x)=H_{g_{t_i}(x)}\to H_x$ in the Hausdorff topology as $i\to\infty$. Since $H_x$ lies in the unstable submanifold and by the choice of $x$ and $t_i$, $g_{t_i}(H_x)\to \{x\}$ as $t_i\to -\infty$. This implies that  $H_x=\{x\}$. Moreover this equality holds for almost every $x$ (by the choice of $x$), hence by continuity again, $H_x=\{x\}$ for any $x$. It follows that $H_{\R}=\operatorname{Id}$, a contradiction to the fact that $X_0\neq 0$ and $t_0^2+s_0^2\neq 0$. Therefore our proof is complete.
\end{proof}

\begin{proof}[Proof of Proposition \ref{pro:cent}]
Fix $q>0$. Notice that by the definition of $\operatorname{GR}(U)$ it follows that chains of length $1$ (i.e.\ trivial chains) do not contribute to the number $\operatorname{GR}(U)$ (see Lemma \ref{lem:chainbasis}). Moreover the chain coming from the $\mathfrak{sl}_2$ triple, i.e.\ $V\mapsto X\mapsto U$ contributes $3$ to the number $\operatorname{GR}(U)$. Since, by assumptions, $\operatorname{GR}(U)>3$ it follows that there exists a chain $X_n\mapsto X_{n-1}\mapsto \cdots\mapsto X_1\mapsto X_0$ of length at least $2$ different from the chain $V\mapsto X\mapsto U$, i.e.\ $U\neq X_0$. Let $\psi_t$ be the flow on $G/\Gamma$ generated by $X_0$; and let $\varphi_t$ be the flow generated by $X_1$. Notice that for every $t_0\in [-2,2]$ the automorphism $\psi_{t_0}$ is uniformly continuous on $G/\Gamma$ (for the right invariant metric $d$). Set
$$
{\bf D}:=\{\psi_{t_0}: |t_0|\leq 2 \}\subset Aut(G/\Gamma,\mu)\cap UC(G/\Gamma,d),
$$
and let ${\bf \tilde{D}}:=\{\psi_1,\psi_{-1}\}\subset {\bf D}\setminus \{Id\}$. Notice that ${\bf \tilde{D}}$ is compact in the uniform topology. We now define Let
\be\label{eq:ak}
A_k(x)=\varphi_{1/k}(x) \text{ for every } x\in G/\Gamma.
\ee
It follows that $A_k\to\operatorname{Id}$ uniformly as $k\to +\infty$.
Fix $\epsilon\in(0,1)$ and $N\in \N$ and let $\kappa:=\epsilon^2$ and $\delta=\frac{\epsilon^4}{10N}$. Notice that by Lemma \ref{lm:ergo} it follows that for every $t_0,s_0\in \R$, with $t_0^2+s_0^2>0$, the automorphism $\psi_{t_0}u_{s_0}$ is ergodic. This, together with the definition of ${\bf \tilde{D}}$, give the ergodicity of the automorphisms in $\bf{C1}$ and $\bf{C2}$ in the definition of $C(\mathbf{D},q)$-property.

For $y\in G/\Gamma$, let $y'=A_k(y)$. By Lemma \ref{lem:chainbasis}, we get:
\be\label{eq:som}
\begin{aligned}
\exp(tU)\exp\left(\frac{1}{k}X_1\right)\exp(-tU)&=\exp\left(\operatorname{Ad}_{\exp(tU)}\left(\frac{1}{k}X_1\right)\right)\\
&=\exp\left(\exp(\operatorname{ad}_{tU})(\frac{1}{k}X_1)\right)\\
&=\exp\left((\operatorname{Id}+\operatorname{ad}_{tU})(\frac{1}{k}X_1)\right)\\
&=\exp\left(\frac{1}{k}X_1+\frac{t}{k} X_0\right).
\end{aligned}
\ee
 Let $M:=k$. Then we have
$$M=k\geq\frac{1}{\delta}=\frac{5N}{\epsilon^4}>\frac{1}{\epsilon^4}=\kappa^{-2}.$$
Since $\epsilon\in(0,1)$, the above inequality also implies that $M\geq N$.
Take $M''\in[\kappa^{-2},M]$. By the definition of $M$ it follows that
$$0<\frac{M''}{k}\leq 1.$$
Let $S_{M''}:=\psi_{-\frac{M''}{k}}$, then by the above bound we have that $S_{M''}\in {\bf D}$ and moreover $S_M=\psi_{-1}\in \mathbf{\tilde{D}}$. Take $t\in [M'',M''+\kappa M'']$. By \eqref{eq:som}, \eqref{eq:ak} and since $[U,X_0]=0$, we have
\begin{equation}
\begin{aligned}
d(u_t(y),\psi_{-\frac{M''}{k}}u_t(y'))&\leq d_G\left(\exp(tU)y,\exp(-\frac{M''}{k}X_0)\exp(tU)\exp(\frac{1}{k}X_1)y\right)\\
&=d_G(e,\exp(-\frac{M''}{k}X_0)\exp(tU)\exp\left(\frac{1}{k}X_1\right)\exp(-tU))\\
&=d_G(e,\exp(-\frac{M''}{k}X_0)\exp(\frac{1}{k}X_1+\frac{t}{k}X_0))<\epsilon,
\end{aligned}
\end{equation}
where the last inequality follows from $\frac{t-M''}{k}\leq \kappa \frac{M''}{k}\leq \kappa$.
Therefore for every $t\in [M'',M''+\kappa M'']$, we have that $d(u_t(y),S_{M''}u_t(y'))<\epsilon$ and hence in particular that the measure estimate in {\bf C1} holds. This finishes the proof.
\end{proof}





Proposition \ref{pro:cent} motivates the following remark:
\begin{remark}In a recent paper, \cite{KVW2}, the authors showed that if $u_t$ is a unipotent flow on $G\slash \Gamma$ with $\operatorname{GR}(U)>3$, then $u_t$ is not loosely Kronecker (without the irreducibility assumption on the lattice). From \cite{Rat1} it follows that every factor $v_t$ of $u_t$ is also unipotent on $G'\slash \Gamma'$. Let $V\in Lie (G')$ be the generator of $v_t$. If $\Gamma$ is additionally irreducible, then
$\operatorname{GR}(V)>3$ (there are no $\operatorname{SL}(2,\R)$ factors). In this case, by \cite{KVW2} it follows that every factor of $u_t$ is not loosely Kronecker. Up to now, all flows  satisfying the strong R-property (see Section \ref{sec:rp}) are loosely Kronecker. The above reasoning shows that any flow as in Section \ref{sec:rp} does not have a common factor with $u_t=\exp(tU)$, where $\operatorname{GR}(U)>3$ (and $\Gamma$ is irreducible). However, as shown in \cite{Rud}, disjointness (which we prove in this setting) is strictly stronger than not having a common factor.
\end{remark}
 The following natural questions arise:
\begin{itemize}
\item \textbf{Question 1:} Does strong R-property imply loosely Kronecker property?
\item \textbf{Question 2:} Is there any other example beyond unipotent flows satisfying $C(q)$-property? In particular, does it hold for time changes of unipotent flows?
\end{itemize}
\begin{remark}
Question 2 is interesting since the main mechanism in obtaining R-property for time changes of horocycle flows is controlling deviations of ergodic averages of the time-changed function (see eg. \cite{FlaminioForni}). Hence establishing the $C(q)$-property for time changes of general unipotent flows would probably require estimates on  ergodic averages of unipotent flows, which is currently not known. This is the main reason we only prove $C(q)$-property for classical unipotent flows (and not their time-changes).
\end{remark}

\section{Flows with strong R-property}\label{sec:rp}
In this section we will prove strong R-property for certain parabolic flows.
We will focus on three classes of flows: horocycle flows  in constant curvature case and their smooth time changes, horocycle flows in variable curvature and (non-trivial) time changes of bounded type Heisenberg nilflows. A version of Ratner's property was established for other classes of parabolic flows (see \cite{Fel-Orn}, \cite{Fr-Lem}, \cite{Fr-Lem2}, \cite{AK2}, \cite{FK}, \cite{KK},  \cite{KLU}, \cite{ForK}). Moreover, one can construct {\em rank one} systems satisfying Ratner's property (the Chacon transformation being a classical example).
It seems that in fact all systems considered in the above quoted papers  satisfy the strong R-property. We focus here only on the three classes mentioned above. We will prove the strong R-property for each class in a separate subsection.

\subsection{Horocycle flows and their smooth time changes}
We will first consider time changes of horocycle flows acting on unit tangent bundle of a surface of constant negative curvature  (including horocycle flows by taking trivial time changes). Next, we will consider horocycle flows on unit tangent bundle of a surface of {\em variable} negative curvature.
\bigskip

\subsubsection{Time changes of horocycle flows over compact surface of constant curvature}\label{sec.tco}
Let $h_t$ denote the horocycle flow on $\operatorname{SL}(2,\R)/\Gamma$ with $\Gamma\subset\operatorname{SL}(2,\R)$ is a cocompact lattice and, for $\tau\in C^1(\operatorname{SL}(2,\R)/\Gamma)$, $\tau>0$, let $h_t^\tau$ denote the time change of $h_t$ given by $\tau$. We will show that strong R-property holds for $h_t^{\tau}$. Recall that the map $\exp:\mathfrak{sl}(2,\R)\to \operatorname{SL}(2,\R)$ is a smooth bijection close to $0\in \mathfrak{sl}(2,\R)$. It follows that for two points $x,y$ which are sufficiently close, we can always write $y=\exp(aU)\exp(bX)\exp(cV)x$ for some (small) $a,b,c$. We will use this decomposition in the following results. We will start with following result, which is an immediate consequence of Proposition 4.1. and Remark 4.2. in \cite{KLU}:
\begin{lemma}\label{lem:dsf}
For every $D\geq1$ and every $\epsilon\in(0,D^{-3})$,  there exists $N_{\epsilon}>0$ and $\delta'>0$ such that for every $x$ and every $y=\exp(aU)\exp(bX)\exp(cV)x$, with $|a|,|b|,|c|\leq \delta'$, we have
$$
d(h^\tau_tx,h^\tau_{\chi_{x,y}(t)}y) \leq  \frac{\epsilon}{2},\;\; \text{where } \chi_{x,y}(t)=e^{-2b}t-e^{-3b}ct^{2},
$$
for every $t\in \R$ satisfying $N_\epsilon\leq |t|\leq D\cdot\min\{|b|^{-1},|c|^{-1/2}\}$.
\end{lemma}
\begin{proof}
Let $K=D$ and pick $\epsilon^2$ instead of $\epsilon$ in Proposition 4.1 in \cite{KLU}, then there exists $N_{\epsilon^2}>0$ and $\delta'>0$ such that for $|a|,|b|,|c|\leq \delta'$ and $|t|\in[N_{\epsilon^2}, D\cdot|c|^{-1/2}]$:
\begin{equation}\label{eq:thcf}
d(h^\tau_tx,h^\tau_{\chi_{x,y}(t)+A_x(t)}y) \leq \epsilon^2,
\end{equation}
where $A_x(t)$ is defined as $\chi_{x,y}(\alpha(x,t))=\alpha(y,\chi_{x,y}(t)+A_x(t))$ (Here, $\alpha(x,t)$ is defined in \eqref{eq:time-change}).

By Lemma 4.1 in \cite{KLU}, we have for every $t=O(\min\{|b|^{-1},|c|^{-1/2}\})$:
\begin{equation}
\left|\int_0^{\alpha(x,t)}(\tau-\tau\circ g_s)(h_tx)dt\right|\leq O(\epsilon^3).
\end{equation}
This inequality together with Proposition 4.1 in \cite{KLU} with $\epsilon^2$, we have the following estimate:
$$|A_x(t)|=O(\epsilon^2).$$

Combining \eqref{eq:thcf} and the above estimate, we have for $|t|\in[N_{\epsilon^2}, D\cdot|c|^{-1/2}]$:
$$d(h^\tau_tx,h^\tau_{\chi_{x,y}(t)}y) \leq \frac{\epsilon}{2}.$$
At the end we let $N_{\epsilon}:=N_{\epsilon^2}$. This finishes the proof.
\end{proof}

We will need also the following classical lemma, which is based on the fact that any two norms on a finite dimensional space are equivalent:
\begin{lemma}\label{lem:pol} There exists a constant $c_0$ such that for every polynomial $p(t)=\beta_1+\beta_2t+\beta_3t^2$ and every  $T>0$, if
$\sup_{t\in [0,T]}|p(t)|\leq c_0$ then
$$
|\beta_1|\leq 1/4,\;\;\; |\beta_2|\leq \frac{1}{4T}\;\;\; |\beta_3|\leq \frac{1}{4T^2}.
$$
\end{lemma}

From the above lemmas we deduce the following:

\begin{proposition}\label{pro:horo1}
 $h_t^\tau$ has the strong R-property.
\end{proposition}
\begin{proof}Let $p=c_0$, where $c_0$ is the constant from Lemma \ref{lem:pol}. We will show that the strong $R(p)$-property holds. Fix $\epsilon\in (0,1)$ and let $\kappa=\kappa(\epsilon):=\min\{\epsilon^{40},N_{\epsilon}^{-20}, \delta'\}$, where $\delta'>0$ and $N_\epsilon$ come from Lemma \ref{lem:dsf} for $D=1$. Let $\delta=\kappa^{10}$ and let $Z=M$. Take $x,x'\in M$ with $d(x,x')<\delta$ and $x$ not in the orbit of $x'$ (which implies that $b^2+c^2>0$). Then $x'=\exp(aU)\exp(bX)\exp(cV)x$, where $\max\{|a|,|b|,|c|\}<2\delta$.
Since $2\delta=2\kappa^{10}\leq\delta'^{10}<\delta'$ it follows that we can use Lemma \ref{lem:dsf} for $x,x'$ to get that for $|t|\in\left[\kappa^{-2},\min\{|b|^{-1},|c|^{-1/2}\}\right]$ (notice that $\kappa^{-2}\geq N_\epsilon$), we have
\begin{equation}\label{eq:fc}
d(h^\tau_tx,h^\tau_{\chi_{x,x'}(t)}x') \leq \frac{\epsilon}{2}.
\end{equation}
Define $f(t)=f_{x,x'}(t):=\chi_{x,x'}(t)-t=(e^{-2b}-1)t-e^{-3b}ct^{2}$ for $t\in[\kappa^{-2},\min\{|b|^{-1},|c|^{-\frac{1}{2}}\}]$. We will show the following:
\begin{enumerate}
\item[${\bf f1.}$] for every  $t\in [\kappa^{-2},\min\{|b|^{-1},|c|^{-\frac{1}{2}}\}]$ and every $s\in [0,\kappa t]$, we have
$$
|f(t)-f(t+s)|\leq \epsilon^2,
$$
\item[${\bf f2.}$] there exists $t_0\in  [\kappa^{-2},\min\{|b|^{-1},|c|^{-\frac{1}{2}}\}]$ such that $|f(t_0)|\geq c_0$.
\end{enumerate}
WLOG we can assume that $t_0$ is the smallest number satisfying this property.
Notice that the strong R-property is then a straightforward consequence of ${\bf f1}$ and ${\bf f2}$. Indeed, it is enough to define $M:=t_0$, $p_L:=f(L)$ for $L\in [\kappa^{-2},M]$. By ${\bf f2}$ it follows that $p_M\geq c_0\geq p/2$. Moreover by ${\bf f1}$  and \eqref{eq:fc}, for every $L\in [\kappa^{-2},M]$ and every $t\in [L,L+\kappa L]$, we get
\begin{multline}
d(h_t^{\tau}(x),h_{t+p_L}^{\tau}(x'))=d\left(h_t^{\tau}(x),h_{t+f(t)+(f(L)-f(t))}^{\tau}(x'))\right)=d\left(h_t^{\tau}(x),h_{\chi_{x,x'}(t)}^{\tau}(x'))\right)\\+{\rm O}(\epsilon^2)\leq \epsilon/2+\epsilon/2=\epsilon.
\end{multline}
This finishes the proof of the strong R-property.
So it only remains to show ${\bf f1}$ and ${\bf f2}$. By definition,
$$
|f(t)-f(t+s)|=(e^{-2b}-1)s-e^{-3b}c(-2st-s^2).
$$
Moreover, $|(e^{-2b}-1)s|\leq 3|b|\kappa t\leq 3\kappa<\epsilon^3$ and similarly $|e^{-3b}c(-2st-s^2)|\leq 5|c|\kappa t^2\leq 5\kappa<\epsilon^3$. This finishes the proof of ${\bf f1}$. For  ${\bf f2}$
notice first that for every $t\in [0,\kappa^{-2}]$, we have $|f(t)|\leq |(e^{-2b}-1)t|+|e^{-3b}ct^2|\leq 4\delta\kappa^{-2}+4\delta\kappa^{-4}\leq \epsilon$ by the definition of $\delta>0$. Therefore if {\bf f2} doesn't hold, i.e.\ $|f(t)|\leq c_0$ for every $t\in [0,\min\{|b|^{-1},|c|^{-\frac{1}{2}}\}]$, then by Lemma \ref{lem:pol}, the coefficients of $f$ have to satisfy
$$|e^{-2b}-1|\leq\frac{1}{4\min\{|b|^{-1},|c|^{-\frac{1}{2}}\}}=\frac{1}{4}\max\{|b|,|c|^{\frac{1}{2}}\},$$
and
$$|e^{-3b}c|\leq\frac{1}{4\min\{|b|^{-2},|c|^{-1}\}}=\frac{1}{4}\max\{|b|^2,|c|\}.$$
Since $|e^{-2b}-1|\geq |b|$,  $|e^{-3b}c|\geq\frac{|c|}{2}$, thus we get
$$|b|\leq\frac{1}{4}|c|^{\frac{1}{2}}\text{ and }\frac{1}{2}|c|\leq\frac{1}{4}|b|^2,$$
which leading to
\begin{equation}\label{eq:contradiction}
2|c|\leq\frac{1}{16}|c|.
\end{equation}
Since $x$ and $x'$ are not in the same orbit of the horocycle flow, we obtain $b^2+c^2>0$. This together with  $|b|\leq\frac{1}{4}|c|^{\frac{1}{2}}$ imply that $|c|>0$, which guarantees that  \eqref{eq:contradiction} is a contradiction. This finishes the proof.
\end{proof}

\bigskip
\subsubsection{Horocycle flows over
compact surfaces of variable
negative curvature}

In this subsection we assume that $v_t$ is the horocycle flow with the uniform parametrization on the tangent space of a negatively curved surface (with variable curvature), see Section \ref{sec:var}. In this section we will use a slightly different representation of two nearby points than in Section \ref{sec.tco}.  More precisely using again the fact that $\exp$ is locally a smooth bijection, we can represent two sufficiently close points $x,y\in M$ as $y=g_av_bk_c x$. The use of this decomposition is not necessary but it makes the computations and formulas more compact.

We state here several crucial lemmas which are based on \cite{Fel-Orn}:

 \begin{lemma}\label{lm:variahoroConLem}There exists $\gamma\in (0,1/2)$ and $C_0>0$ such that for every $\epsilon>0$, there exists $\kappa'=\kappa'(\epsilon)>0$ and $\delta'=\delta'(\epsilon)>0$ such that for every $x\in M$ and every $y=g_av_bk_cx$ with $|a|,|b|,|c|<\delta'$, we have
 \be\label{eq:FO1}
 d(v_sx,v_{e^{a}\sigma(x,s,c)}y)<\epsilon \mbox{ for every } s\in \left[0, \frac{\delta'^2}{2|c|}\right],
 \ee
 where  $\sigma(\cdot,\cdot,\cdot):M\times \R\times \R\to \R$ satisfies the following:
 \begin{itemize}
 \item[$(i)$] {\em scaling property}: for every $r\in \R$, $ \sigma(x,e^rs,e^{-r}t)=e^r\sigma(g_{-r}x,s,t)$;
 \item[$(ii)$] for every $|\kappa''|<\kappa'$, we have
 \be\label{eq:FO2}
 |(\sigma(x,s,c)-s)-(\sigma(x,(1+\kappa'')s,c)-(1+\kappa'')s)|\leq \epsilon |\sigma(x,s,c)-s|,
 \ee
 and
 \be\label{eq:FO3}|\sigma(x,\frac{s}{2},t)-\frac{s}{2}|\leq(\frac{1}{2}-\gamma)|\sigma(x,s,t)-s|;
 \ee
 \item[$(iii)$] if $|sc|$ is small enough, then $\sigma(x,s,c)-s$ is monotone in $s$.
 \item[$(iv)$] $\sigma(x,0,c)=0$.
 \end{itemize}
 \end{lemma}
\begin{proof}
The first part of the above lemma, i.e.\ \eqref{eq:FO1} follows from Lemma 3.8. in \cite{Fel-Orn}, $(i)$ follows from geometric definition of $\sigma$ in \cite{Fel-Orn} and renormalization property of the flow $u_t$, $(ii)$ follows from Lemma 3.7 in \cite{Fel-Orn}, $(iii)$ follows from Corollary 3.5. in \cite{Fel-Orn} and $(iv)$ also follows from geometric definition of $\sigma$ in \cite{Fel-Orn}.
\end{proof}

\begin{lemma}\label{lm:variahoroConLem2} Fix $\epsilon>0$ and let
$\delta'=\delta'(\epsilon)>0$ be the constant from Lemma \ref{lm:variahoroConLem}. There exists $\delta\in(0,\delta')$ such that if $y=g_av_bk_cx$, $c\neq0$ and $|a|,|b|,|c|<\delta$, then there exists $N>0$ such that $N|c|<\frac{\delta'^2}{2}$ and
\be\label{eq:hjk}
|e^a\sigma(x,N,c)-N|=1.
\ee
Moreover, if $N$ is the smallest number satisfying \eqref{eq:hjk}, then $N$ satisfies: $N\geq\frac{1}{|a|}$,
\begin{equation}\label{eq:controlofM}
e^a|\sigma(x,N,c)-N|\leq \frac{3}{2}\gamma^{-1},
\end{equation}
and
\begin{equation}\label{war}
|(e^a-1)N|\leq1+\frac{3}{2}\gamma^{-1}.
\end{equation}
\end{lemma}
\begin{proof}
The proof of this lemma is similar to the proof of Lemma 3.9. in \cite{Fel-Orn}, we give a full proof here for completeness.

Denote $r(t)=|e^a\sigma(x,t,c)-t|$, $A(t)=\sigma(x,t,c)-t$, then we have $$r(t)=|e^aA(t)+(e^a-1)t|.$$
Notice that by \eqref{eq:FO3}, since $\gamma<1/2$ and by the definition of $A(t)$ it follows that
$A(t)\geq \frac{1}{1/2-\gamma} A(t/2)$.
Therefore if $\delta$ is small enough we can guarantee that for $|c|<\delta$, we have $|A(\frac{\delta'^2}{2|c|})|>4\gamma^{-1}$. Let $B=\frac{\delta'^2}{2|c|}$ to simplify the notation. Notice that
by \eqref{eq:FO3} and the definition of $A(t)$ if $r(B)\leq 1$, then we have
\begin{equation}\label{eq:eqqw}
\begin{aligned}
r\left(\frac{B}{2}\right)&\geq\left|(e^a-1)\frac{B}{2}\right|-\left|e^aA\left(\frac{B}{2}\right)\right|\\
&\geq\left|(e^a-1)\frac{B}{2}\right|-(\frac{1}{2}-\gamma)|e^aA(B)|\\
&\geq\gamma|e^aA(B)|-\frac{1}{2}r(B)\geq\gamma|e^aA(B)|-\frac{1}{2}>1.
\end{aligned}
\end{equation}
This implies that there is some number $N\in [0,B]$ such that $r(N)=1$. We WLOG assume that $N$ is the smallest number satisfying $r(N)=1$. Notice that reasoning analogously to \eqref{eq:eqqw} with $B=N$, we get
$$1\geq r\left(\frac{N}{2}\right)\geq \gamma|e^aA(N)|-\frac{1}{2}.$$
Therefore,
$$|e^aA(N)|\leq\frac{3}{2}\gamma^{-1},$$
which gives \eqref{eq:controlofM} by the definition is $A(t)$.

Then the definition of $r(t)$ and above inequality imply $$|(e^a-1)N|\leq(1+\frac{3}{2}\gamma^{-1}),$$
which gives \eqref{war}.

Finally, notice that:
\begin{equation}
\begin{aligned}
1&=|e^a\sigma(x,N,c)-N|=|e^aA(N)-(1-e^a)N|\\
&\geq|e^aA(N)|-|1-e^a|N\\
&=\frac{3}{2}\gamma^{-1}-|1-e^a|N\\
&\geq \frac{3}{2}\gamma^{-1}-2|a|N.
\end{aligned}
\end{equation}
Since $0<\gamma<1/2$, we have $N\geq \frac{1}{|a|}$.
\end{proof}

Using the above lemmas, we have the following:

\begin{proposition}\label{pro:horo2}
$v_t$ has the strong R-property.
\end{proposition}
\begin{proof} We will show that the strong R-property holds with $p=1$.
Fix $\epsilon\in(0,1)$ and let $\kappa:=\frac{1}{2}\min\{\kappa'(\frac{\gamma\epsilon}{3}),\gamma\epsilon^{20}\}$. By taking  $\kappa$ smaller if necessary, we can WLOG assume that if $|r|<10\kappa$, then for every $x\in M$, $d(u_rx,x)<\epsilon/10$.

Let $\delta>0$ come from Lemma \ref{lm:variahoroConLem2} for $\delta'=\delta'(\frac{\gamma\epsilon}{3})$ where $\delta'$ comes from from Lemma \ref{lm:variahoroConLem}). By taking smaller $\delta$ if necessary, we can WLOG assume that for every $x\in M$ and every $x'=g_av_bk_cx$ if $d(x,x')<\delta$ then $|a|,|b|,|c|<\min\{\delta',\kappa^{10}\}$. Take any $x, x'\in M$ so that $d(x,x')<\delta$ and $x'$ is not in the orbit of $x$ (which implies that $a^2+c^2>0$).

\noindent\textbf{First case: $c=0$}. In this case, the proof is similar to Lemma 3.9. in \cite{Fel-Orn}, we give a full proof here for completeness.

If $c=0$, then $a\neq0$ and $x'=g_av_bx$. Applying $v_{e^{a}t}$ to both sides and also using the renormalization property of uniform parametrization, we have
$$v_{e^{a}t}x'=g_av_{b+t}x.$$
Since $\delta$ is small enough (and so $|a|,|b|<\delta'$), we get
\begin{equation}\label{eq:w=0}
d(v_{e^{a}t}x',v_tx)<\frac{\epsilon}{2}.
\end{equation}

Let $M:=|1-e^a|^{-1}\geq \frac{1}{2|a|}$. Then, for any $L\in[\kappa^{-2},M]$ and $t\in[L,(1+\kappa)L]$, we have:
\begin{equation}
|L(e^{a}-1)|\leq M|e^{a}-1|=1,
\end{equation}
and
\begin{equation}
|(t-L)(e^{a}-1)|\leq \kappa L|1-e^{a}|\leq 2\kappa.
\end{equation}
This implies that for $t\in[L,(1+\kappa)L]$, $|e^at-[t+(e^a-1)L]|\leq 2\kappa$. Define $p_L:=(e^{a}-1)L$. Then $p_M=1$ and  by \eqref{eq:w=0}, for $t\in[L,(1+\kappa)L]$,  we have
$$
d(v_tx,v_{t+p_L}x')\leq d(v_tx,v_{e^at}x')+d(v_{e^at}x',v_{t+p_L}x')\leq \epsilon/2+\epsilon/2=\epsilon
$$
This gives the strong R-property in the case $c=0$.
\\

\noindent\textbf{Second case: $c\neq0$}.

If $c\neq0$, then by Lemma \ref{lm:variahoroConLem} for $x,x'$ such that $d(x,x')<\delta$, we have for $t\leq\frac{\delta'^2}{2|c|}$ :
\begin{equation}\label{eq:5.1epsilon/3}
d(v_tx,v_{e^a\sigma(x,t,c)}x')<\frac{\gamma\epsilon}{3}.
\end{equation}

Define $g(t)=g_{x,x'}(t):=e^{a}\sigma(x,t,c)-t$ for $t\in[\kappa^{-2},\frac{\delta'^2}{2|c|}]$. We will show the following:
\begin{enumerate}
\item[${\bf g1.}$] there exists $t_0\in  [\kappa^{-2},\frac{\delta'^2}{2|c|}]$ such that $|g(t_0)|= 1$.
\item[${\bf g2.}$] let $M$ be the smallest number satisfying {\bf g1}. Then for every  $t\in [\kappa^{-2},M]$ and every $s\in [0,\kappa t]$, we have
$$
|g(t)-g(t+s)|\leq \frac{2\epsilon}{3},
$$
\end{enumerate}

Now we will show how {\bf g1} and {\bf g2} implies the strong R-property. Denote $p_L=e^{a}\sigma(x,L,c)-L$. Then by {\bf g2} we have $|p_M|\geq\frac{1}{2}$ and $|p_L|\leq1$ for $L\in[\kappa^{-2},M]$. Moreover, by {\bf g1} and \eqref{eq:5.1epsilon/3}, for every $L\in[\kappa^{-2},M]$ and every $t\in[L,L+\kappa L]$, we get
\begin{equation}
d(v_{t}x,v_{t+p_L}x')\leq d(v_tx,v_{e^a\sigma(x,t,c)}x')+ d(v_{e^a\sigma(x,t,c)}x',v_{t+p_L}x')<\epsilon.
\end{equation}
This finishes the proof of strong R-property when $c\neq0$. So it only remains to show {\bf g1} and {\bf g2}. In fact, notice that {\bf g1} directly follows from Lemma \ref{lm:variahoroConLem2}. Therefore we only need to show {\bf g2}.

As for {\bf g2}, due to $\kappa^{-2}\leq t\leq M$ and $M\leq\frac{\delta'^2}{2|c|}$, then by Lemma \ref{lm:variahoroConLem} $(iii)$, we know $e^a(\sigma(x,t,c)-t)$ is monotone for $t\in[\kappa^{-2}, M]$. Then notice that $e^a(\sigma(x,0,c)-0)=0$ and thus we have $e^a|\sigma(x,t,c)-t|\leq e^a|\sigma(x,M,c)-M|$. Recall the choice of $M$ and \eqref{eq:controlofM}, we have for $t\in[\kappa^{-2}, M]$:
\be\label{eq:eat}e^a|\sigma(x,t,c)-t|\leq\frac{3}{2}\gamma^{-1}.
\ee

Then for $t\in[\kappa^{-2},M]$ and $s\in[0,\kappa t]$ we have,
\begin{equation}
\begin{aligned}
&|g(t)-g(t+s)|=|[e^{a}\sigma(x,t,c)-t]-[e^{a}\sigma(x,t+s,c)-(t+s)]|\\
=&|e^a[(\sigma(x,t,c)-t)-(\sigma(x,t+s,c)-(t+s))]-(e^a-1)s|\\
\leq& e^a|(\sigma(x,t,c)-t)-(\sigma(x,t+s,c)-(t+s))|+|e^a-1|\kappa t\\
\leq&\frac{\gamma\epsilon}{3}e^a|\sigma(x,t,c)-t|+|e^a-1|\kappa M\\
\leq& \frac{\epsilon}{2}+(1+\frac{3}{2}\gamma^{-1})\kappa<\frac{2\epsilon}{3},
\end{aligned}
\end{equation}
where line $3$ to line $4$ is due to $(ii)$ in  Lemma \ref{lm:variahoroConLem} for $s=t$ and $\kappa''=\frac{s}{t}\leq \kappa\leq \kappa'$ and line $4$ to line $5$ is due to \eqref{eq:eat} and \eqref{war} in  Lemma \ref{lm:variahoroConLem2} (recall that $N=M$). This finishes the proof of {\bf g2}.
\end{proof}

\subsection{Time changes of bounded type Heisenberg nilflows}

We will work with a special flow $T_t^\tau$, where the rotation frequency is of bounded type. First, we have the following general lemma to deal with the strong R-property for special flows:
\begin{lemma}\label{lm:specialstrongR}
Suppose $T$ is an ergodic automorphism acting on a compact metric space $(X,\mathcal{B},\mu,d)$ and let $f\in C(X,\mathcal{B},\mu)$ be a positive function which is bounded away from zero. Let $T_t^{f}$ be the corresponding special flow and $q>0$. Assume that for every $\epsilon'>0$, there exists $\kappa'\in(0,\epsilon')$, $Z'\in\mathcal{B}$, $\mu(Z')\geq1-\epsilon'$ and $\delta'>0$ such that for every $x,x'\in Z$ with $d(x,x')<\delta'$ there exists $M'\in\mathbb{N}$ with $M'\geq 8C_0^2\kappa'^{-2}$ (where $C_0=\max\{2,3(\int_Xfd\mu)^{-\frac{1}{2}}\}$) such that  the following holds:
\begin{itemize}
  \item[\textbf{R1'}.] for every $L'\in[\kappa'^{-2},M']$, there exists $q_{L'}\in[-q,q]$ such that
  \begin{itemize}
  \item[a.] for every $n\in[L',L'(1+\kappa')]$,
  $$
   d(T^nx,T^nx')<\epsilon'\;\text{ and }\;  |S_{n}(f)(x)-S_{n}(f)(x')-q_{L'}|<\epsilon';
   $$
   \item[b.] $|q_{M'}|>q/2$.
  \end{itemize}

\end{itemize}
Then the special flow $(T_t^{f})_{t\in\mathbb{R}}$ has the strong R-property.
\end{lemma}
\begin{proof}We will show that the strong $R(p)$-property holds with $p=q$. We assume for simplicity that $\int_X f d\mu=1$. Fix $\epsilon>0$ and let $C_{\min}=\min\left\{2,\min_{x\in X}f(x)\right\}>0$. We apply the conditions of Lemma \ref{lm:specialstrongR} with $\epsilon':=\min\{\frac{\epsilon}{3},\frac{C_{\min}\epsilon}{16}\}$. By Ergodic Theorem, for every $\epsilon,\kappa'>0$, there exists $N_0=N(\epsilon,\eta)$ and a set $A=A(\epsilon,\eta)$ with $\mu(A)>1-\epsilon/10$ such that for every $R\geq N_0$ and every $x\in A$, we have
\begin{equation}\label{eq:ergodicEst}
\left|\frac{1}{R}S_R(f)(x)-1\right|\leq\frac{\kappa'}{300}.
\end{equation}

Define
$$
X(\epsilon):=\left\{(x,s)\in X^f:\frac{C_{\min}\epsilon}{8}<s<f(x)-\frac{C_{\min}\epsilon}{8}\right\}.
$$
Notice that for every interval $[a,b]$, with $b-a\geq \sup_{x\in X} f(x)$ and every $(x,s)\in X^f$, we have
\be\label{eq:badX}
\left|\{t\in [a,b]\;:\; T_t^f(x,s)\in X(\epsilon)\}\right|\geq (1-\frac{\epsilon}{4})(b-a).
\ee

Let $\kappa:=\min\left\{N_0^{-20},\frac{\kappa'^{20}}{C_0}\right\}$. Since $f$ is a continuous function it follows that there exists $\delta=\delta(\epsilon)\in [0,\min\{\delta',\epsilon'\}]$ such that for any two points $x,x'\in X$ satisfying $d(x,x')<\delta$  and every $n\in[0,\kappa^{-3}]$, we have
\be\label{eq:contg}
|S_n(f)(x)-S_n(f)(x')|<\epsilon.
\ee
Define $Z:=\{(x,s)\in X^f\;:\; x\in A\cap Z'\}\subset X^f$ (where $Z'$ comes from Lemma \ref{lm:specialstrongR}). Notice that by absolute continuity of integrals, up to changing $\epsilon>0$ if necessary, it follows that $\mu^f(Z)\geq 1-\epsilon$. Take $(x,s), (x',s')\in Z$ with $d^f((x,s),(x',s'))\leq \delta$ and $x\neq x'$. By the definition of $Z$ it follows that $x,x'\in Z'$ and so {\bf R1'} holds for $x$ and $x'$. Let $M'=M'(x,x')$ be as in {\bf R1'}. Define
\begin{equation}\label{eq:defml}
M:=\frac{f^{(M')}(x)-s}{2}.
\end{equation}
Notice that since $|q_{M'}|\geq q/2$, by \eqref{eq:contg} it follows that $M'\geq \kappa^{-3}$. Therefore
\begin{equation}
\begin{aligned}
M&\geq \frac{f^{(M')}(x)-s}{2}\geq \frac{1}{2}C_{\min}(M'-1)\geq \kappa^{-2},
\end{aligned}
\end{equation}
since $\kappa$ is small. Take $L\in [\kappa^{-2}, M]$ and let $t\in [L,L+\kappa L]$. Recall that $N(x,s,t)\in \N$ is defined by $T^f_t(x,s)=(T^{N(x,s,t)}x,s+t- S_{N(x,s,t)}(f)(x))$. By definition, we have
$$
[\sup_{x\in X}f(x)]N(x,s,t)\geq  S_{N(x,s,t)}(f)(x)\geq t\geq \kappa^{-2}.
$$
By the definition of $\kappa$ and since $(x,s)\in Z$, it follows that \eqref{eq:ergodicEst} holds for $R=N(x,s,t)$ and $x\in A$. Therefore,
\begin{equation}
\begin{aligned}
\left(1-\frac{\kappa'}{300}\right)N(x,s,t)&\leq S_{N(x,s,t)}(f)(x)\leq t+s\leq (1+\kappa)L+s\leq(1+\kappa)(L+s)\\
&\leq(1+\kappa)S_{N(x,s,L)+1}(f)(x)\\
&\leq(1+\kappa)\left(1+\frac{\kappa'}{300}\right)(N(x,s,L)+1)\\
&\leq(1+\kappa')\left(1-\frac{\kappa'}{300}\right)N(x,s,L),
\end{aligned}
\end{equation}
where the last inequality is due to the choice of $\kappa$ and $\kappa'$. The above inequality implies that $N(x,s,t)\leq (1+\kappa')N(x,s,L)$. On the other hand, due to monotonicity of $N(x,s,t)$ in $t$, we have $N(x,s,t)\geq N(x,s,L)$. As a result, we have
\be\label{eq:qe}N(x,s,L)\leq N(x,s,t)\leq (1+\kappa')N(x,s,L)
\ee
for $t\in[L,(1+\kappa)L]$ with $L\in[\kappa^{-2},M]$. Assume additionally that $t\in[L,(1+\kappa)L]$ is such that $T_t^f(x,s)\in X(\epsilon)$. Notice that by {\bf R1'} for $L'=N(x,s,L)$, using \eqref{eq:qe} and $|s-s'|<\delta$, we get
\begin{equation}
\begin{aligned}
S_{N(x,s,t)}(f)(x')&\leq S_{N(x,s,t)}(f)(x)-q_{N(x,s,L)}+\epsilon'< t+s-\frac{C_{\min}\epsilon}{8}-q_{N(x,s,L)}+\epsilon'\\
&\leq t+s'-q_{N(x,s,L)}+\delta-\frac{C_{\min}\epsilon}{8}+\epsilon',
\end{aligned}
\end{equation}
and
\begin{equation}
\begin{aligned}
S_{N(x,s,t)+1}(f)(x')&\geq S_{N(x,s,t)+1}(f)(x)-q_{N(x,s,L)}-\epsilon'>t+s+\frac{C_{\min}\epsilon}{8}-q_{N(x,s,L)}-\epsilon'\\
&\geq t+s'-q_{N(x,s,L)}-\delta+\frac{C_{\min}\epsilon}{8}-\epsilon'.
\end{aligned}
\end{equation}
Due to our choice of $\delta$, $\epsilon$ and $\epsilon'$, we have $\delta-\frac{C_{\min}\epsilon}{8}+\epsilon'\leq0$ and thus above inequalities imply that:
\begin{equation}
S_{N(x,s,t)}(f)(x')\leq t+s'-q_{N(x,s,L)}\leq S_{N(x,s,t)+1}(f)(x').
\end{equation}
By definition of the special flow, we have $$T^f_{t-q_{N(x,s,L)}}(x',s')=(T^{N(x,s,t)}x',t+s'-q_{N(x,s,L)}-S_{N(x,s,t)}(f)(x')),$$ and thus
\begin{equation}\label{eq:ann}
\begin{aligned}
&d^f(T^f_t(x),T^f_{t-q_{N(x,s,L)}}(x'))\\&=d(T^{N(x,s,t)}x,T^{N(x,s,t)}x')+|s-S_{N(x,s,t)}(f)(x)-s'+q_{N(x,s,L)}+S_{N(x,s,t)}(f)(x')|\\
&\leq d(T^{N(x,s,t)}x,T^{N(x,s,t)}x')+|s-s'|+|S_{N(x,s,t)}(f)(x)-S_{N(x,s,t)}(f)(x')-q_{N(x,s,L)}|\\
&\leq\epsilon'+\delta+\epsilon'\leq\epsilon.
\end{aligned}
\end{equation}

Notice that by \eqref{eq:badX}, the measure of the set  $\{t\in[L,L(1+\kappa)]:T^f_t(x,s)\in X(\epsilon)\}$ is at least $(1-\epsilon)\kappa L$.
We set $p_L:=q_{N(x,s,L)}$. Then, by \eqref{eq:ann}, for any $L\in[\kappa^{-2},M]$, we have:
$$|\{t\in[L,L(1+\kappa)]:d^f(T^f_t(x),T^f_{t-p_L}(x'))\}|\geq(1-\epsilon)\kappa L$$
and $|p_M|=|q_{N(x,s,M)}|=|q_{M'}|\geq\frac{q}{2}$, which finishes the proof of strong R-property of special flow $T_t^f$.
\end{proof}

In order to use the above lemma to prove the strong R-property of time change of Heisenberg nilflows, we use the following crucial lemma from \cite{ForK}:

\begin{lemma}[Lemma 5.6, \cite{ForK}]\label{le:est} Let $f\in W^s(\T^2)$ with $s>7/2$. There exists $C_1=C_{\alpha,f}>0$ such that for any $(x,y)$ and $(x',y')$, for every $n\in\mathbb N$, $$|S_n(f)(x,y)-S_n(f)(x',y')|\le C_1(|y-y'|n^{1/2}+|x-x'|n^{3/2}).$$
\end{lemma}

We will also need the following result
\begin{proposition}[Proposition 6.1., \cite{ForK}]\label{propfork} For any $\alpha$ of bounded type and for any $f\in W^s(\T^2)$, $s>7/2$, there exists a constant $D_{\alpha,f}>0$ such that the following holds. For every $(x,y,),(x',y')\in \T^2$ let $T:=\min\{|x-x'|^{-2/3},|y-y'|^{-2}\}$. There exists $n_0=n_0(x,y,x',y')\in [0,D_{\alpha,f}T]$ such that
$$
|S_{n_0}(f)(x,y)-S_{n_0}(f)(x',y')|>1.
$$
\end{proposition}

\begin{proposition}\label{pro:heis}
If $\tau\in W^s(\mathbb{T}^2)$ is nontrivial, $s>\frac{7}{2}$ and $\alpha$ is of bounded type, then $T_t^\tau$ has the strong R-property.
\end{proposition}
\begin{proof}

We will use Lemma \ref{lm:specialstrongR}. Fix $q=1$ and fix $\epsilon>0$. Let $\kappa=\kappa(\epsilon)=\epsilon^{10}$, $Z=\T^2$ and $\delta=\kappa^{10}$. Take $(x,y),(x',y')\in \T^2$ with $d((x,y),(x',y'))\leq \delta$.  For $n\in [0,D_{\alpha,f}T]$ (from Proposition \ref{propfork}), let
\be\label{eq:an}
a_n:=S_n(f)(x,y)-S_n(f)(x',y').
\ee
Notice first that $T^n(x,y)=(x+n\alpha, y+nx+\frac{n(n-1)}{2}\alpha)$ and, by the definition of $T$ in Proposition \ref{propfork} and the definition of $\delta$, it follows that for every  $n\in [0,D_{\alpha,f}T]$, we have
\be\label{eq:aiso}
d(T^n(x,y),T^n(x',y'))<\epsilon.
\ee
Notice that for $n\in [0,D_{\alpha,f}T]$ and $m\in [n,(1+\kappa)n]$, we have by cocycle identity
$$
|a_m-a_n|= |S_{n-m}(f)(T^n(x,y))-S_{n-m}(f)(T^n(x',y'))|
$$
and $T^n(x,y)=(x+n\alpha,y+nx+\frac{n(n-1)}{2}\alpha)$. So by Lemma \ref{le:est} and since $m-n\leq \kappa n$, we have
\begin{multline}\label{eeas}
|a_m-a_n|\leq C_1[|y-y'+n(x-x')|n^{1/2}+|x-x'|n^{3/2}]\leq \\2C_1\kappa^{1/2}\left[|y-y'|D^{1/2}_{\alpha,f}T^{1/2}+|x-x'|D^{3/2}_{\alpha,f}T^{3/2}\right]<\epsilon^2,
\end{multline}
By the definition of $T$ and $\kappa$.

Let $M'\in [0,D_{\alpha,f}T]$ be the smallest number  such that $|a_{M'}|\geq 1$ (notice that by Proposition \ref{propfork} such $M'$ always exists). Moreover, since $|x-x'|,|y-y'|<\delta$, by Lemma \ref{le:est} and the definition of $\delta$ it follows that $M'\geq \kappa^{-2}$.
Take $L'\in [\kappa^{-2},M']$ and let $q_{L'}=a_{L'}$. By the definition of $M'$ it follows that $|q_{L'}|\leq 1$. By \eqref{eeas} it follows that for $m\in[L',(1+\kappa)L']$, we have
$$
|a_m-a_{L'}|\leq \epsilon^2
$$
and moreover $a_{M'}\geq 1-\epsilon^2>1/2$. This together with \eqref{eq:aiso} gives {\bf R1'} and finishes the proof of Proposition \ref{pro:heis}.
\end{proof}

\section{Proof of Theorem \ref{thm:main}}

Theorem \ref{thm:main} now follows from Theorem \ref{thm:dis}, Proposition \ref{pro:cent}, Proposition \ref{pro:horo1}, Proposition \ref{pro:horo2} and Proposition \ref{pro:heis}.

More precisely, Proposition \ref{pro:cent} shows that a unipotent flow $u_t$ defined on $G/\Gamma$ induced by $U\in\mathfrak{g}$ with $\operatorname{GR}(U)>3$ satisfies $C(q)$-property for every $q>0$; Proposition \ref{pro:horo1} shows that a time change of horocycle flow on $\operatorname{SL}(2,\mathbb{R})/\Gamma$ has strong R-property if $\tau\in C^1(\operatorname{SL}(2,\mathbb{R})/\Gamma)$ and $\tau>0$; Proposition \ref{pro:horo2} guarantees that horocycle flow $v_t$ defined on a compact negatively variable curvature surface with uniform parametrization has strong R-property and Proposition \ref{pro:heis} shows that a time change $T_t^{\tau_1}$ of bounded type Heisenberg nilflow satisfies $\tau_1\in W^s(\mathbb{T}^2)$ for $s>\frac{7}{2}$, then $T_t^{\tau_1}$ has strong R-property. Combining these results with Theorem \ref{thm:dis}, we know that $u_t$ is disjoint with $h_t^{\tau}$, $v_t$ and $T_t^{\tau_1}$, respectively.


\end{document}